\documentclass[a4paper,12pt,leqno]{article}
\usepackage{fullpage}
\usepackage{times, amsmath, amssymb, amsthm, amscd, bbm, latexsym, inputenc}
\usepackage[english]{babel}
\usepackage{color}
\usepackage{graphicx}
\usepackage{mathrsfs}
\usepackage[all]{xy}
\usepackage[pdftex,pagebackref=true,colorlinks=true,linkcolor=blue,citecolor=blue,urlcolor=blue]{hyperref}
\usepackage{currfile} 

\theoremstyle{plain}

\newtheorem{theorem}{Theorem}[section]

\newtheorem{thm}{Theorem}[section]
\newtheorem{prop}[theorem]{Proposition}

\theoremstyle{definition}
\theoremstyle{definition}

\newtheorem{remarks}[theorem]{Remarks}
\newtheorem{rem}[theorem]{Remark}
\numberwithin{equation}{section}
\numberwithin{figure}{section}
\numberwithin{table}{section}

\newcommand{\C}{\mathbb{C}}

\newcommand{\M}{\mathbb{M}}

\newcommand{\Q}{\mathbb{Q}}

\newcommand{\R}{\mathbb{R}}

\newcommand{\Z}{\mathbb{Z}}

\newcommand{\cM}{\mathcal{M}}

\newcommand{\cO}{\mathcal{O}}

\def\ms{\mathscr}

\def\bbm{\mathbbm}
\def\span{\operatorname{span}}
\def\M{\operatorname{M}}
\def\wt{\widetilde}
\providecommand{\abs}[1]{\left\vert#1\right\vert}

\providecommand{\bmat}[1]{\begin{bmatrix}#1\end{bmatrix}}

\providecommand{\bmat}[1]{\begin{bmatrix}#1\end{bmatrix}}

\def\build#1#2\fin{\mathrel{\mathop{\kern0pt#1}\limits#2}}

\newcommand{\divv}{\mathop{\rm div}\nolimits}

\newcommand{\grad}{\mathop{\rm grad}\nolimits}
\newcommand{\End}{\mathop{\rm End}\nolimits}
\newcommand{\Gal}{\mathop{\rm Gal}\nolimits}

\title{One can't hear orientability of surfaces}

\author{Pierre B\'erard  and David L. Webb}

\date{\today}

\begin{document}
\maketitle

\begin{abstract}
The main result of this paper is that one cannot hear orientability of a surface with boundary. More precisely, we construct two isospectral flat surfaces with boundary with the same Neumann spectrum, one orientable, the other non-orientable. For this purpose, we apply Sunada's  and Buser's methods in the framework of orbifolds. Choosing a symmetric tile in our construction, and adapting a folklore argument of Fefferman, we also show that the surfaces have different Dirichlet spectra. These results were announced in the {\it C. R. Acad. Sci. Paris S\'er. I Math.}, volume 320 in 1995, but the full proofs so far have only circulated in preprint form.
\end{abstract}

MSC~2010: {58J50, 58J32}\\[5pt]

\tableofcontents

\section{Introduction}

Let $M$ be a compact Riemannian manifold with boundary. The {\it spectrum\/} of $M$ is the sequence of eigenvalues of the Laplace-Beltrami operator $\Delta  f= -\divv(\grad f)$ acting on smooth functions on $M$; when $\partial M \ne \varnothing$, one can impose either Dirichlet boundary conditions on the function $f$ (i.e., $f|_{\partial M}=0$) or Neumann boundary conditions (the normal derivative $\partial f/\partial n$ vanishes on $\partial M$). Mark Kac's classic paper \cite{K} has stimulated a great deal of interest in the question of what geometric or topological properties of $M$ are determined by its spectrum. In \cite{GWW}, it was shown that, in Kac's terminology, one cannot hear the shape of a drum, or of a bell:  that is, there exist pairs of nonisometric planar domains that have the same spectra, for either Dirichlet (in the case of a drum) or Neumann (in the case of a bell) boundary conditions. The construction uses an adaptation to orbifolds of Sunada's technique for constructing isospectral manifolds. For many other examples, see \cite{BCDS}.

This paper uses orbifold techniques to exhibit examples of pairs of Neumann isospectral flat surfaces with boundary, one of which is orientable and the other nonorientable. Performing the same construction using a tile with an additional symmetry, we exhibit Neumann isospectral bordered surfaces that are not Dirichlet isospectral by adapting an argument of C.~Fefferman.

To our knowledge, these are the first confirmed examples in which Neumann isospectrality holds but Dirichlet isospectrality fails.

This paper is organized as follows. Section~\ref{IsospectralManifolds} briefly reviews the Sunada construction of isospectral manifolds; this is used to construct the Neumann isospectral surfaces $M_1$ and $M_2$ in section~\ref{Construction}. Section~\ref{Transplantation} summarizes some representation-theoretic calculations that furnish a computation of the most general ``transplantation'' map, which transplants a Neumann eigenfunction on $M_1$ to a Neumann eigenfunction on $M_2$ with the same eigenvalue; as in \cite{BCDS} and \cite{Be3}, transplantation of eigenfunctions affords an elementary visual proof of the isospectrality. In section~\ref{Fefferman}, we give an unpublished argument of C.~Fefferman; while in section~\ref{NeumannNotDirichlet}, by modifying our construction slightly, we show that there are Neumann isospectral flat surfaces with boundary $M_1$ and $M_2$ that are not Dirichlet isospectral by adapting Fefferman's argument.  Section~\ref{InaudSing} contains some concluding observations.

Since the constructions and proofs are quite elementary, the exposition is aimed at a general reader and is essentially self-contained, although reference to \cite{GWW} may be helpful.  These results were announced in \cite{BW} and circulated in an MSRI preprint \cite{BW1}; this paper is a revision of \cite{BW1} and contains the details of the results announced in \cite{BW} (along with some improvements), after a long delay.

\medskip
\emph{Acknowledgments.}~ We wish to thank Carolyn Gordon for helpful discussions, Dorothee Sch\"uth for carefully reading a preliminary version, and Peter Doyle for discussing aspects of this work.  We are grateful to Bob Brooks for communicating Fefferman's argument. The first author acknowledges the hospitality of IMPA, where some of this research was conducted.  The second author is grateful to MSRI for its support and for its congenial atmosphere.  Both \cite{BW} and \cite{BW1} acknowledged support from NSF grants DMS-9216650 and DMS 9022140, from CNRS (France), and  from CNPq (Brazil). The authors wish to thank the referee for his comments.

\section{Isospectral manifolds}\label{IsospectralManifolds}

Let $G$ be a finite group, and let $\Gamma$ be a subgroup of $G$. Then the left action of $G$ on the coset space $G/\Gamma$ determines a linear representation $\C[G/\Gamma]$ of $G$, where $\C[G/\Gamma]$ denotes a complex vector space with basis $G/\Gamma $; the $G$-module $\C[G/\Gamma]$ can be viewed as the representation $(\bbm{1}_\Gamma)\negthickspace\uparrow^G_\Gamma=\C[G]\otimes_{\C[\Gamma]}\C$ of $G$ induced from the one-dimensional trivial representation $\bbm{1}_{\Gamma}$ of the subgroup $\Gamma$.

Now let $\Gamma _1$ and $\Gamma _2$ be subgroups of $G$. Then $(G,\Gamma_1,\Gamma _2)$ is called a {\it Gassmann-Sunada triple\/} if  $\C[G/\Gamma_1]$ and $\C[G/\Gamma_2]$ are isomorphic representations of $G$. The formula for the character of an induced representation (see, e.g., \cite{Se}, section 7.2 or \cite{JamesLiebeck}, 21.19) shows that isomorphism of the induced representations $(\bbm{1}_{\Gamma_1})\negthickspace\uparrow_{\Gamma_1}^G$ and $(\bbm{1}_{\Gamma_2})\negthickspace\uparrow_{\Gamma_2}^G$ is equivalent to the assertion that $\Gamma _1$ and $\Gamma _2$ are {\it elementwise conjugate\/} or \emph{almost conjugate} subgroups of $G$ --- that is, there exists a bijection $\Gamma _1 \to \Gamma _2$ carrying each element $\gamma \in\Gamma _1$ to a conjugate element $g\gamma g^{-1} \in \Gamma _2$, where the conjugating element $g\in G$ may depend upon $\gamma $. See R.~Perlis \cite{P} or the papers of R.~Guralnick (e.g., \cite{Gu}) for many examples.

Elementwise conjugate subgroups that are not conjugate as subgroups were used by Gassmann \cite{Ga} to exhibit pairs of nonisomorphic number fields having the same zeta function and hence the same arithmetic. The analogy of Galois theory with covering space theory led T.~Sunada to apply Gassmann-Sunada triples to develop a very powerful technique for constructing isospectral Riemannian manifolds that are not isometric, as follows.  (Conjugate subgroups are trivially almost conjugate; however, in that case the two Riemannian manifolds arising from Sunada's Theorem are isometric.  Thus we seek nonconjugate pairs of almost-conjugate subgroups.)

\begin{theorem}[Sunada \cite{Su}] Let $M$ be a compact Riemannian manifold with boundary, and let $G$ be a finite group acting on $M$ by isometries. Suppose that $(G,\Gamma _1,\Gamma _2)$ is a Gassmann-Sunada triple, with $\Gamma _1$ and $\Gamma _2$ acting freely on $M$. Then the quotient manifolds $M_1 = \Gamma _1\backslash M$ and $M_2 = \Gamma _2\backslash M$ are isospectral. (If $\partial M \ne \varnothing$, then either Dirichlet or Neumann boundary conditions can be imposed.)
\end{theorem}

For surveys of some of the extensions and applications of Sunada's technique, see \cite{Sunada20YrsLater}, \cite{Brooks}, \cite{BrooksMonthly}, and \cite{BerardPesce}.

P.~B\'{e}rard \cite{Be2} gave a representation-theoretic proof of Sunada's theorem, relaxing the requirement that the subgroups $\Gamma _1$ and $\Gamma _2$ act freely; the conclusion is then that the orbit spaces $\cO_1 = \Gamma _1\backslash M$ and $\cO_2 = \Gamma _2\backslash M$ are isospectral as orbifolds. Recall that an $n$-dimensional {\it orbifold\/} is a space whose local models are orbit spaces of $\R^n$ under action by finite groups $G$; an orbifold with boundary is similarly modeled locally on quotients of a half-space by finite group actions. See \cite{T}, \cite{Sc}, or \cite{ALR} for more details. The {\it singular set\/} of the orbifold consists of all points where the isotropy is nontrivial.

For our purposes, an understanding of one of the simplest examples of an orbifold with boundary will suffice. Consider the rectangle $R =[-1,1]\times [0,1]$ in $\R^2$. The group $\Gamma  = \Z/2\Z$ acts via the reflection $(x,y)\mapsto(-x,y)$ about the vertical axis, and the quotient orbifold $\cO = \Gamma \backslash R$ is, as a point set, the square $[0,1]\times [0,1]$. However, $\cO$ has a singular set consisting of a distinguished ``mirror edge'' $\cM = \{0\} \times [0,1]$, the image of the fixed-point set of $\Gamma $. At points not in $\cM$, the local structure is that of $\R^2$ or of the half-plane, and the projection $R \overset{\pi}{\to} \cO$ is locally a double cover; however, at points of the mirror edge $\cM$, the local structure is that of $\R^2$ modulo a reflection, and the orbit map $\pi $ is a covering map only in the sense of orbifolds.  To distinguish $[0,1]\times [0,1]$ as a point set from $\cO = \Gamma \backslash R$ viewed as an orbifold with boundary, we write $|\cO|$ for the underlying space $[0,1]\times [0,1]$ of $\cO$ when the orbifold structure is disregarded. It is important to note that $\partial\cO$ (the orbifold boundary) differs from $\partial|\cO|$ in that $\partial \cO$ does {\it not} contain the mirror edge $\cM$. In the orbifold sense, the fundamental group of $\cO$ is $\Z/2\Z$, and $R$ is its universal cover. Indeed, the loop in $\cO$ consisting of the straight-line path from $(1/2,1/2)$ to $(0,1/2)$ and back again lifts to a non-closed path, as shown below.
\begin{center}
\includegraphics[width=4.0in]{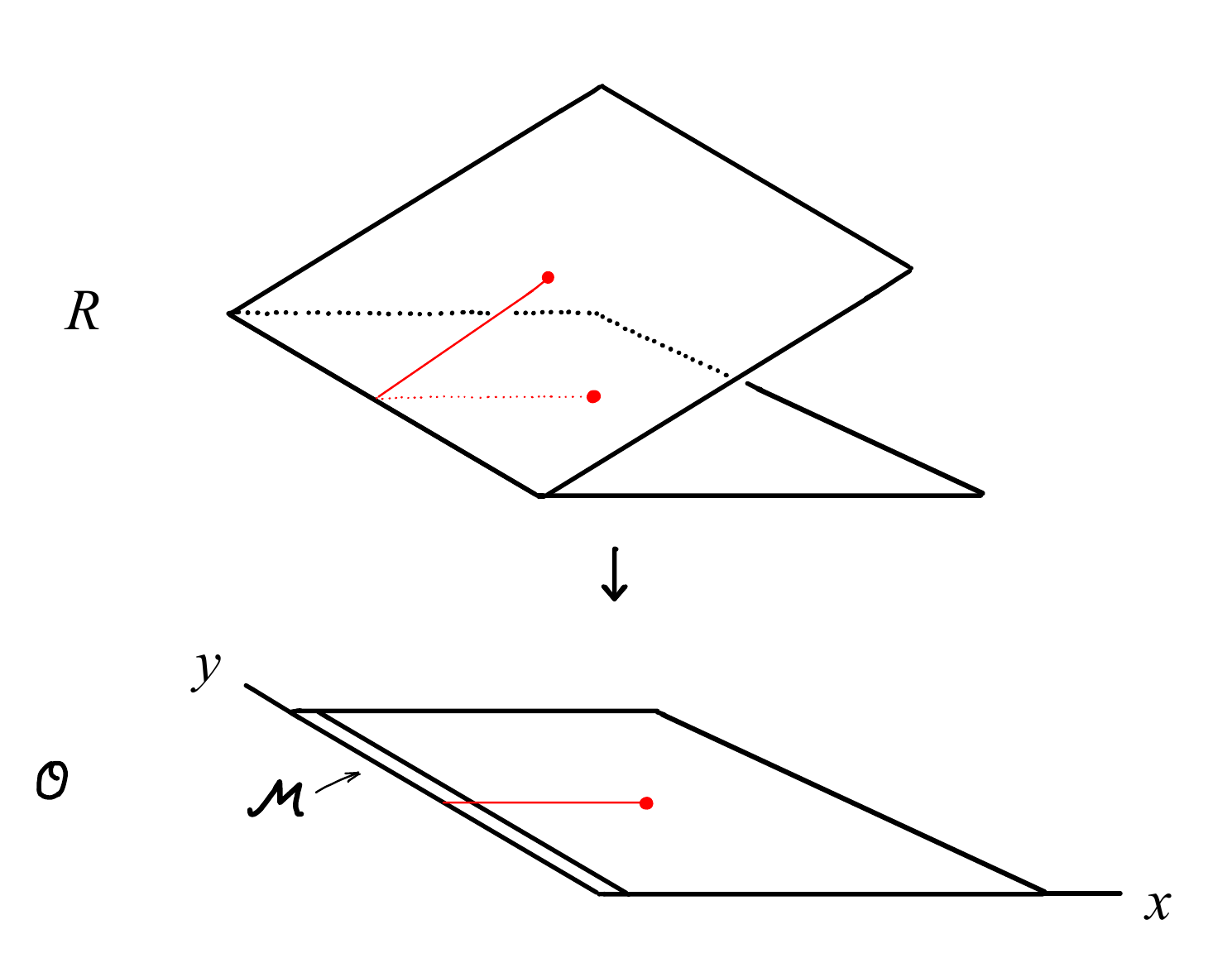}
\end{center}

Mirror edges will be denoted in our drawings by doubled lines.

The smooth functions on $\cO = \Gamma \backslash R$ are precisely the smooth functions on $R$ that are $\Gamma $-invariant, and the {\it spectrum\/} of $\cO$ is the spectrum of the Laplacian acting on $\Gamma $-invariant functions on $R$. Note that a $\Gamma $-invariant smooth function on $R$ restricts to a function on $|\cO|$ with zero normal derivative on the mirror edge $\{0\} \times [0,1]$; thus the Dirichlet spectrum of the orbifold $\cO$ is the spectrum of the domain $|\cO|$ with mixed boundary conditions: Dirichlet conditions on the three edges $[0,1] \times \{0\}$, $[0,1] \times \{1\}$, and $\{1\} \times [0,1]$ forming the orbifold boundary $\partial \cO$, but Neumann conditions on the mirror edge $\cM$.

\begin{rem}\label{RmkOnBoundaryConds}  In general, for a 2-orbifold $\cO$ with boundary whose singular set consists of disjoint mirror arcs, the Dirichlet spectrum of $\cO$ is simply the spectrum of the underlying space $|\cO|$ with mixed boundary conditions: Neumann conditions on the mirror arcs, but  Dirichlet conditions on the remainder of the boundary.  Similarly, the Neumann spectrum of $\cO$ is simply the Neumann spectrum of $|\cO|$, since the reflection-invariance forces Neumann conditions on the part of the boundary of $|\cO|$ corresponding to mirror arcs of $\cO$. This simple observation will be used repeatedly in what follows.
\end{rem}

Given a Gassmann-Sunada triple $(G,\Gamma _1,\Gamma _2)$, in order to apply Sunada's Theorem one needs a manifold $M$ on which $G$ acts by isometries. Perhaps the easiest way to obtain such an $M$ is to begin with a manifold $M_0$ whose fundamental group $\widehat{G} = \pi _1(M_0)$ admits a surjective homomorphism $\varphi:\widehat{G} \to G$. Setting $\widehat{\Gamma} _i = \varphi^{-1} (\Gamma_i)$ for $i=1,2$ defines two subgroups of $\pi_1(M_0)$; by covering space theory, there are covering spaces $M_1$ and $M_2$ corresponding to $\widehat{\Gamma}_1$ and $\widehat{\Gamma}_2$. Also, there is a common regular covering $M$ of $M_1$ and $M_2$ associated to the subgroup $\widehat{K} = \ker \varphi$, so that $M_i = \Gamma_i\backslash M$, $i=1,2$, and $M_0 = G\backslash M$.

P.~Buser \cite{Bu1} exploited the observation that if $\pi _1(M_0)$ is free, then Schreier graphs furnish a concrete means of constructing $M$, $M_1$ and $M_2$ without explicit reference to the universal cover; when $M_0$ is a surface with nonempty boundary, he used this construction to exhibit isospectral flat bordered surfaces.

Given a group $G$, a $G$-set $X$, and a generating set $S$ for $G$, recall that the  {\it Schreier graph\/} of $X$ relative to the generating set $S$ has vertex set $X$, with a directed edge labeled by $s\in S$ joining the vertex $x$ to the vertex $sx$, for each $x\in X$, $s\in S$. Concretely, one chooses the bordered surface $M_0$ to be a thickened one-point union of circles, with one circle for each generator; the case of a three-element generating set $S = \{a,b,c\}$ is depicted below.
\begin{center}
\includegraphics[width=4.5in]{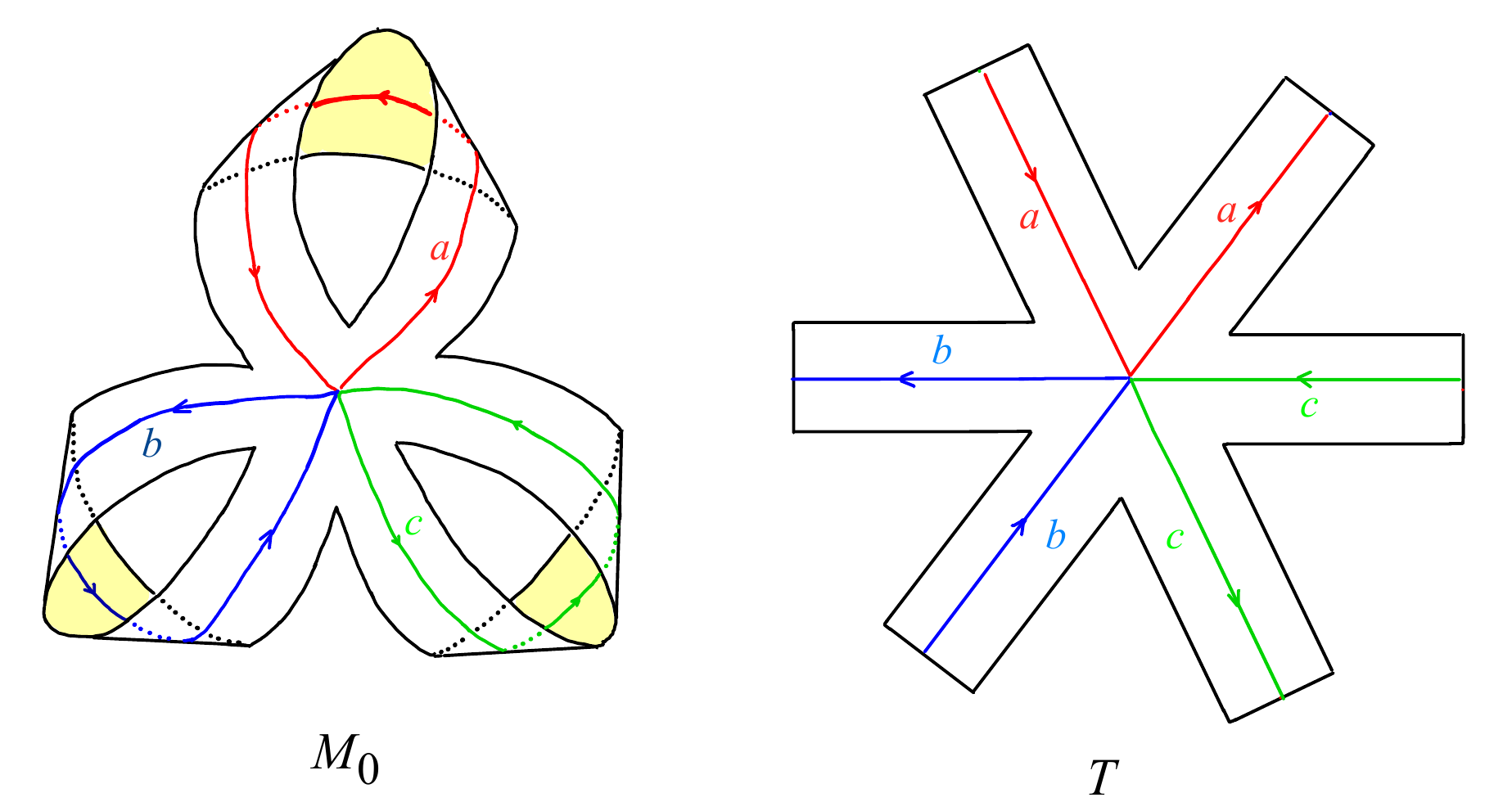}
\end{center}

The construction of $M_i$ for $i=1,2$ now goes as follows. Begin with $[G:\Gamma _i]$ copies of the fundamental domain $T$ depicted above, labeled by the cosets $G/\Gamma_i$. For each generator $s\in S$, glue the edge of the outgoing leg of the tile labeled $x\Gamma _i$ to the edge of the incoming $s$-leg of the tile labeled $sx\Gamma _i$. The surface so constructed is a thickened Schreier graph of $G/\Gamma _i \cong \widehat{G}/\widehat{\Gamma} _i$, so it has fundamental group $\widehat{\Gamma}_i$ and covers $M_0$, and hence must be precisely the manifold $M_i$ defined above.

Suppose now that the group $G$ is generated by a set $S$ of involutions. Then one can construct analogously a pair of isospectral orbifolds with boundary as follows. Let $\cO_0$ be a disk with $[G:\Gamma _i]$ nonintersecting mirror arcs (the case of $S =\{a,b,c\}$ is depicted in Figure~1).
\begin{center}
\includegraphics[width=4.75in]{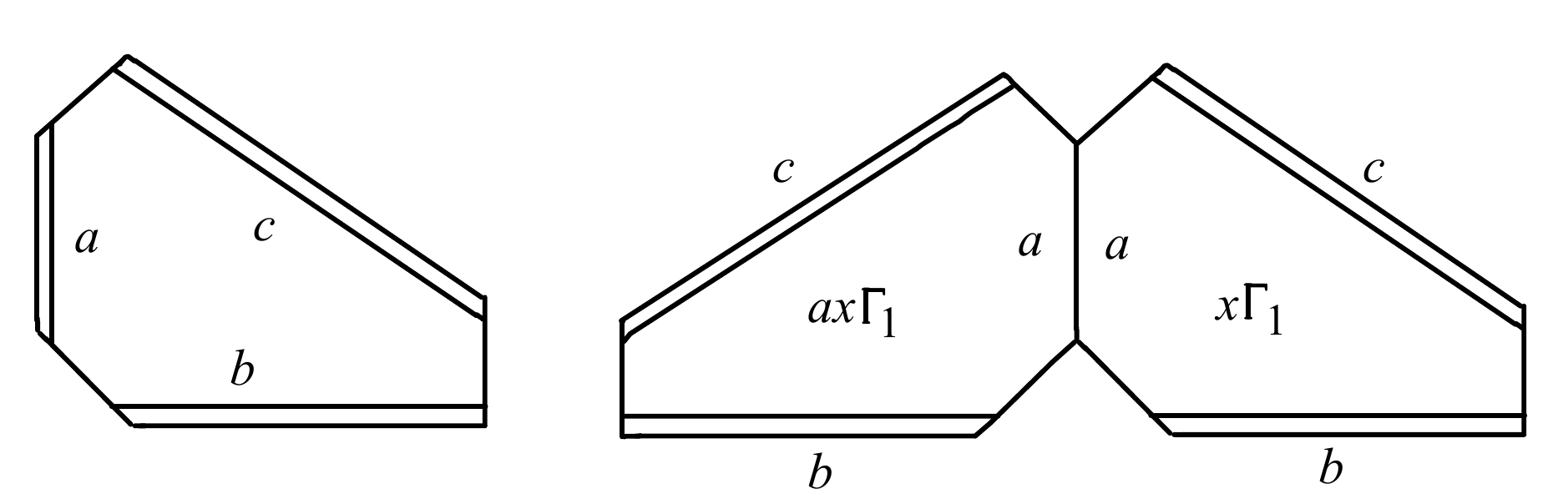}\linebreak
\noindent $\cO_0$\qquad\qquad\qquad\qquad\qquad\qquad\qquad\qquad\qquad\qquad

\vspace{3pt}
{\it Figure 1\hspace{4.5cm} Figure 2}\qquad\qquad\qquad\quad
\end{center}

\noindent Then the orbifold fundamental group $\widehat{G} = \pi_1(\cO_0)$ is the free product of $|S|$ copies of $\Z/2\Z$. Define a surjective homomorphism $\varphi:\widehat{G} \to G$ by sending the obvious ``mirror reflection'' generators of $\Z/2\Z*\Z/2\Z* \cdots *\Z/2\Z$ to the elements of $S$. As before, let $\widehat{\Gamma}_i = \varphi^{-1}(\Gamma_i)$, $i=1,2$. Then by covering space theory, there is an orbifold cover $\cO_i$ of $\cO_0$ corresponding to the subgroup $\widehat{\Gamma}_i\subseteq \widehat{G} = \pi _1(\cO_0)$, which can be constructed explicitly as follows: begin with  $[G:\Gamma _i]$ copies of the fundamental domain $|\cO_0|$, labeled by the elements of $G/\Gamma _i$. For each $s\in S$, glue the mirror edge labeled $s$ of the tile labeled $x\Gamma _i$ to the mirror edge labeled $s$ of the tile labeled $sx\Gamma _i$ so that reflection in the common edge interchanges the two tiles, as in Figure~2; if $s$ fixes a coset, then no identification is performed on that edge. Then the orbifold so constructed has fundamental group $\widehat{\Gamma}_i$, so it must be $\cO_i$. By B\'{e}rard's version of Sunada's Theorem, $\cO_1$ and $\cO_2$ are isospectral orbifolds.

\section{Construction}\label{Construction}

In this section, we turn to the proof of the following.

\begin{theorem}\label{t1}
There exists a pair $M_1$, $M_2$ of flat surfaces with boundary which are Neumann isospectral, yet $M_1$ is nonorientable while $M_2$ is orientable.
\end{theorem}

\begin{proof}
The surfaces $M_1$ and $M_2$ will be constructed as the underlying spaces of Neumann isospectral orbifolds $\cO_1$ and $\cO_2$ with boundary whose singular sets consist of disjoint unions of mirror arcs.  As noted above in Remark~\ref{RmkOnBoundaryConds}, this means that Neumann boundary conditions hold on the mirror edges of the underlying surfaces $M_1 = |\cO_1|$ and $M_2 = |\cO_2|$ as well as on the edges forming the orbifold boundary; since $\partial M_i$ consists of the boundary of $\cO_i$ together with the mirror edges, it follows that Neumann conditions hold on the entire boundary of $M_i$.

The Gassmann-Sunada triple we use was first considered by Gerst \cite{Ge}; it was also used by Buser \cite{Bu2} to construct isospectral Riemann surfaces, and in \cite{GWW}. Let $G$ be the semidirect product of a multiplicatively-written cyclic group $\langle s\rangle$ of order 8 by its full automorphism group; the latter is a Klein 4-group, generated by the automorphism $t$ sending $s\mapsto s^7$ and the automorphism $u$ sending $s\mapsto s^3$. Thus $G$ is the semidirect product $\Z_8 \rtimes \Z^\times_8$, with $s$ generating the cyclic subgroup $\Z_8$; a presentation is $G = \big\langle s,t,u\mid s^8 = t^2 = u^2 = [t,u] = 1,~ tst=s^7,~usu=s^3\big\rangle$. Let $\Gamma_1 = \{1,t,u,tu\}$, $\Gamma_2 = \{1,t,s^4u,s^4tu\}$. Then $(G,\Gamma_1,\Gamma_2)$ is a Gassmann-Sunada triple (with $\Gamma_1$ and $\Gamma_2$ nonconjugate).

Now let $\cO_0$ be the orbifold with boundary depicted below, in Figure 3:
\begin{center}
\includegraphics[width=1.75in]{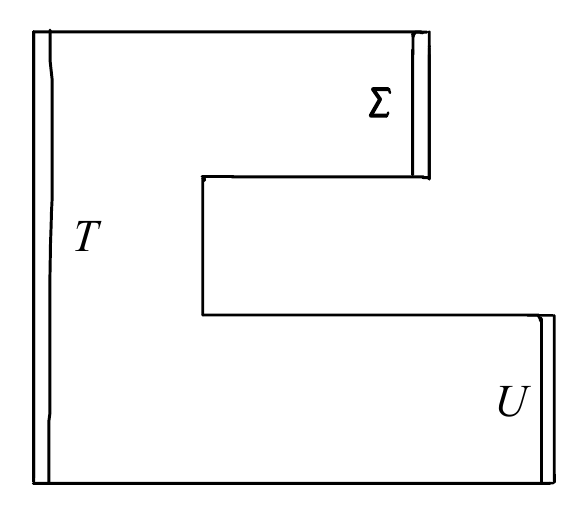}\linebreak
{\it Figure 3}
\end{center}

Its fundamental group $\widehat{G}$ is then given by $\widehat{G} = \pi_1(\cO_0) = \big\langle \Sigma,T,U\mid \Sigma^2 = T^2 = U^2 = 1\big\rangle\cong \Z/2\Z* \Z/2\Z*\Z/2\Z$. Let $\sigma  = st \in G$. Then $\sigma$, $t$ and $u$ are involutions, so we can define $\varphi: \pi_1(\cO_0)\to G$ by $\Sigma\mapsto \sigma$, $T\mapsto t$, $U\mapsto u$. Since $G$ is generated by the set $\{\sigma,t,u\}$, it follows that $\varphi$ is surjective.

As in section \ref{IsospectralManifolds}, let $\widehat{\Gamma}_i = \varphi^{-1}(\Gamma_i)$. The elements $1,s,s^2,\dotsc, s^7$ form a set of coset representatives for $G/\Gamma _i$ and hence for $\widehat{G}/\widehat{\Gamma}_i$ for $i=1,2$; we will denote the coset $s^i\Gamma _1$ or $s^i\Gamma _2$ simply by $i$ in the depiction of the Schreier graphs below. The action on $G/\Gamma_1$  and $G/\Gamma _2$ of the three generators $\sigma $, $t$ and $u$ is most easily recorded by the Schreier graphs below; when dealing with generating sets consisting of involutions, we adopt the convention that a single undirected edge labeled by a generator $r\in S$ replaces the two oppositely-directed edges labeled $r$ and $r^{-1}$ joining a pair of vertices; if an edge labeled $r$ leaves a vertex labeled $x$ and does not terminate at another vertex, this indicates that the generator $r$ fixes the coset $x$.  Thus when dealing with involutive generators, we are replacing a pair of oppositely-directed edges by a single undirected edge, and replacing a loop based at a vertex by a ``half-edge'' emanating from that vertex.

\begin{equation}\label{SchreierGraphs}
\xymatrix@R=-3pt@C=13pt{
&&&&&&&&{}&&{}&&&&&&&&&&{}\\
{2}&&{7}&&{1}&&{0}&&&&&&{2}&&{7}&&{1}&&{0}\\
{\circ}\ar^{\sigma}@{-}[rr]&&{\circ}\ar^{t}@{-}[rr]&&{\circ}\ar^{\sigma}@{-}[rr]&&{\circ}\ar^{t}@{-}[uurr]&&&&&&{\circ}\ar_u@{-}[uull]\ar^{\sigma}@{-}[rr]&&{\circ}\ar^t@{-}@/^/[rr]\ar_u@{-}@/_/[rr]&&{\circ}\ar^{\sigma}@{-}[rr]&&{\circ}\ar^t@{-}[uurr]\\ \\ \\ \\
&&&&&&&&{}\ar^u@{-}[uuuull]&&{}&&&&&&&&&&{}\\ \\ \\
&&&&&&&&&&&&&&&&&& \\ \\ \\
&&&&&&&&{}&&{}&&&&&&&&&&{}\\ \\ \\ \\
{\circ}\ar^t@{-}@/^/[uuuuuuuuuuuuuu]\ar_u@{-}@/_/[uuuuuuuuuuuuuu]\ar_{\sigma}@{-}[rr]&&{\circ}\ar_t@{-}[rr]\ar_>>>{u}@{-}[rruuuuuuuuuuuuuu]
&&{\circ}\ar_{\sigma}@{-}[rr]
\ar@{-}'[luuuuuuu]_<<<u'[lluuuuuuuuuuuuuu]
&&{\circ}\ar^u@{-}[uuuurr]&&&&&&{\circ}\ar^t@{-}[uuuuuuuuuuuuuu]\ar_{\sigma}@{-}[rr]&&{\circ}\ar^u@{-}@/^/[rr]\ar_t@{-}@/_/[rr]&&{\circ}\ar_{\sigma}@{-}[rr]&&{\circ}\ar_u@{-}[uuuuuuuuuuuuuu]\\
{6}&&{3}&&{5}&&{4}&&&&&&{6}&&{3}&&{5}&&{4}\\
\\
&&&&&&&&{}\ar^t@{-}[uuull]&&{}\ar_u@{-}[uuurr]&&&&&&&&&&{}\ar^t@{-}[uuull]
}
\end{equation}
\hspace{1.4in}$G/\Gamma_1$\hspace{2.65in}$G/\Gamma_2$

\vspace{5pt}
Let $\cO_i$ be the orbifold covering of $\cO_0$ corresponding to the subgroup $\widehat{\Gamma}_i$ of $\widehat{G}$. The bordered surfaces $M_1 = |\cO_1|$ and $M_2 = |\cO_2|$ are shown below: the first model of $M_1$ is embedded in $\R^3$, while the second model shows it immersed with an arc of self-intersection; the immersed version enables one to see easily an involutive symmetry that will be exploited in section \ref{NeumannNotDirichlet}.
\begin{center}
\includegraphics[width=5.5in]{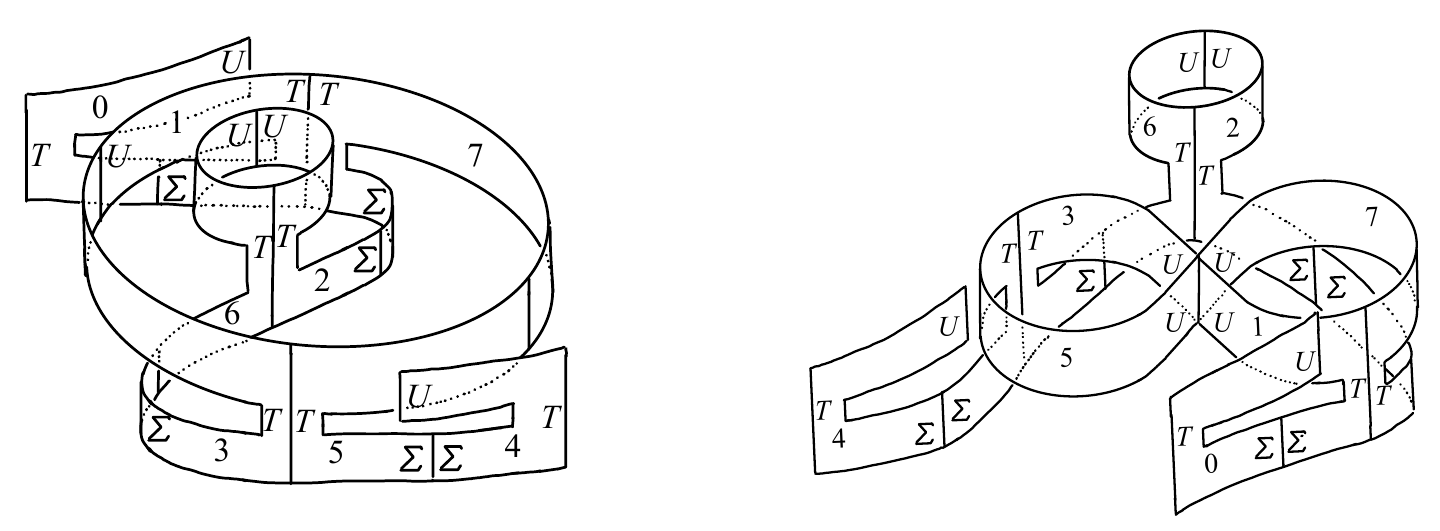}\linebreak
$M_1\text{ embedded}\hspace{2.0in}M_1\text{ immersed}$\linebreak
\includegraphics[width=3.5in]{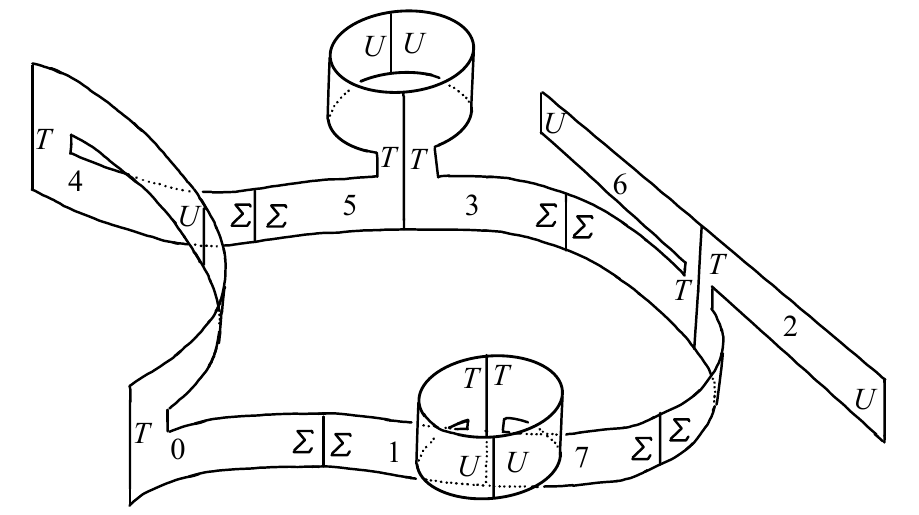}\linebreak
$M_2$\linebreak
{\it Figure 4}
\end{center}

Clearly, $M_1$ is nonorientable while $M_2$ is orientable. Indeed, an embedded M\"obius strip in $M_1$ can be seen in the following picture:
\begin{center}
\includegraphics[width=4.25in]{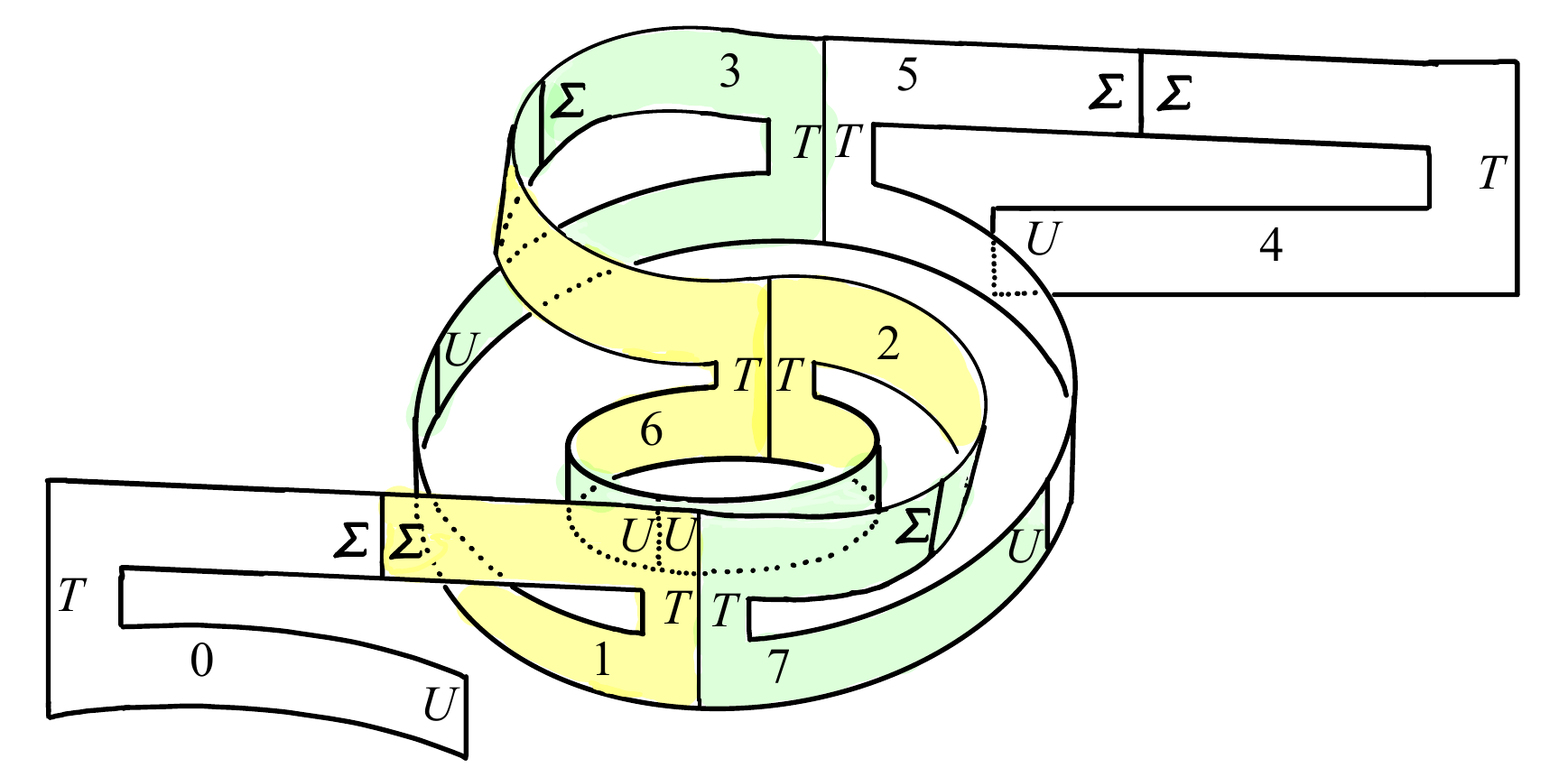}
\end{center}
If one colors one side of the fundamental tile yellow and the other side green, then beginning with the yellow side of tile 1, the succession of the $U$-gluing of tile 1 to tile 3, the $\Sigma$-gluing of tile 3 to tile 6, the $T$-gluing of tile 6 to tile 2, and the $\Sigma$-gluing of tile 2 to tile 7 shows that the $T$-gluing of tile 7 back to tile 1 forces the gluing of the green side of tile 7 to the yellow side of tile 1, exhibiting the union of tiles 1, 3, 6, 2, and 7 as a one-sided surface.  Thus one cannot ``hear'' orientability of surfaces with boundary.
\end{proof}

An immediate consequence of the above result is the following.
\begin{theorem}\label{CStructureInaudible}
One cannot hear whether a surface with boundary admits a complex structure.
\end{theorem}

\begin{proof}
Perform the same construction as above, but using a different fundamental tile: rather than the tile in Figure~3, use a right-angled hyperbolic hexagon, three pairwise nonadjacent sides of which are mirror loci, as in Figure~5:
\begin{center}
\includegraphics[width=2.0in]{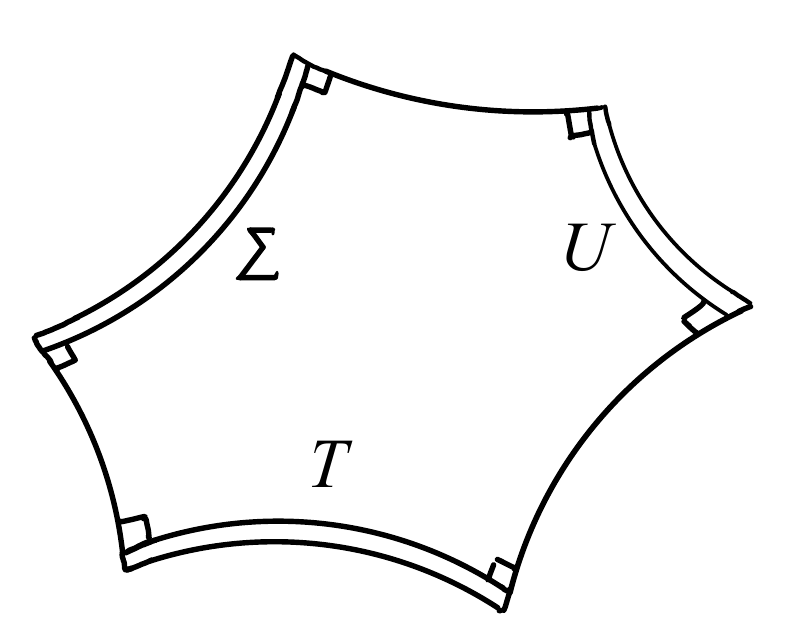}\linebreak
{\it Figure 5}
\end{center}
The resulting manifolds $M_1$ and $M_2$ are homeomorphic to the manifolds of Figure 4 above.  Both are hyperbolic surfaces with piecewise-smooth boundary.  An orientable hyperbolic surface has a Riemann surface structure, by the uniformization theorem (see \cite{FaMa}, \S 11.1.1 or \cite{Hubbard}, Chapter 1).  Thus $M_2$ has a complex structure.  However, $M_1$, being nonorientable, cannot have a complex structure, since any complex manifold is canonically orientable.
\end{proof}

\begin{remarks} We list a few immediate remarks:
\begin{enumerate}

\item The nonorientability of $M_1$ can be seen algebraically from the presence of cycles of odd length in the Schreier graph of $G/\Gamma_1$ in \eqref{SchreierGraphs}.

\item A slight modification of the construction in Theorem \ref{CStructureInaudible} yields a pair of isospectral compact hyperbolic surfaces with smooth boundary whose boundary components are closed geodesics; the first of these surfaces is nonorientable but the second is orientable.  For $i=1,2$, perform the same construction, but using as basic tile the pair of pants (in the terminology of \cite{BuserBook}, section 3.1, a \emph{$Y$-piece}) obtained by gluing together two copies of the right-angled hyperbolic hexagon along the alternating non-mirror edges of the tile of Figure~5.  Since each $Y$-piece is constructed by ``doubling'' a right-angled hyperbolic hexagon by gluing along the mirror edges, this amounts to beginning with two copies of the manifold $M_i$ of Theorem \ref{CStructureInaudible} and gluing the two copies together along the non-mirror edges of their boundaries, thereby effectively doubling the manifold $M_i$ of Theorem \ref{CStructureInaudible} to obtain a hyperbolic surface $N_i$ with geodesic boundary components (the boundary components correspond to half-edges in the Schreier graphs in \eqref{SchreierGraphs}).  For more information on pants decompositions of nonorientable hyperbolic surfaces, see \cite{PP}.

\item Peter Doyle and Juan Pablo Rossetti have shown \cite{DoyleRossetti} that if two \emph{closed} hyperbolic surfaces have the same spectrum, then for every possible length, the two surfaces have the same number of orientation-preserving geodesics and the same number of orientation-reversing geodesics.  Thus one \emph{can} hear orientability of closed hyperbolic surfaces.

\item The reader is encouraged to build paper and tape models; this will make it easy to follow the arguments in sections \ref{Transplantation} and \ref{NeumannNotDirichlet}.
\end{enumerate}
\end{remarks}

\section{Transplantation of eigenfunctions}\label{Transplantation}

In this section, we present an elementary visual proof of the Neumann isospectrality using the idea of ``transplantation'' of eigenfunctions from B\'{e}rard's proof of Sunada's Theorem.  We give an explicit combinatorial recipe for transplanting a Neumann eigenfunction on $M_1$ to a Neumann eigenfunction for the same eigenvalue on $M_2$. To do this we first determine the irreducible representations of $G$, then express the induced representations $\C[G/\Gamma _i] = (\bbm{1}_{\Gamma _i})\negthickspace \uparrow^G_{\Gamma _i}$ in terms of the irreducible representations. The reader who is only interested in the isospectrality proof can skip to Remarks~\ref{MapsOfSubreps}.

Recall that $G$ is a semidirect product $K\rtimes H$, where the normal subgroup $K = \langle s\rangle$ is cyclic of order 8, and $H = \langle t,u\mid t^2=u^2 = [t,u] = 1\big\rangle$ is a Klein 4-group, with $t$ acting as the automorphism $s\mapsto s^7$ and $u$ acting as the automorphism $s\mapsto s^3$. The rational group algebra $\Q[K\rtimes H]$ of a semidirect product has the structure of a trivial crossed-product or {\it twisted group algebra\/} $(\Q K)\sharp H$ (see \cite{CR} for general information on crossed products). The group algebra $\Q C_n$ of a cyclic group $C_n$ of order $n$ decomposes as a product of cyclotomic fields: $\Q C_n \cong \prod_{d|n} \Q[\zeta _d]$, where $\zeta _d$ is a primitive $d$th root of unity (see \cite{L}), so $\Q K \cong \Q \times \Q \times \Q [i] \times \Q [\zeta]$, where $\zeta $ is a primitive eighth root of unity. The $H$-action on $\Q K$ stabilizes this decomposition, so $\Q G \cong (\Q\times \Q \times \Q [i] \times \Q [\zeta ]) \sharp H \cong \Q H\times \Q H\times (\Q[i]\sharp H) \times (\Q [\zeta ]\sharp H)$. Tensoring with $\C$, we obtain
\begin{equation}\label{3.1}
\C G \cong \C H \times \C H \times (\C \otimes_{\Q}\Q[i]) \sharp H\times (\C \otimes_{\Q} \Q [\zeta ]) \sharp H.
\end{equation}

The first two factors of \eqref{3.1} yield eight one-dimensional representations, denoted $\bbm{1}^{abc}$, where each of $a,b,c$ can be +1 or $-1$; each ${\bbm{1}}^{abc}$ is a one-dimensional vector space, where $s$ acts as the scalar $a$, $t$ acts as $b$, and $u$ acts as $c$. We will denote the trivial one-dimensional representation $\bbm{1}^{+++}$ simply as $\bbm{1}$, and the ``parity'' representation $\bbm{1}^{-++}$ simply as $\bbm{1}^-$. In the third factor of \eqref{3.1}, $H$ acts on $\C \otimes \Q[i]\cong  \C \times \C$ as follows: $t$ and $u$ both act by the involution $(x,y) \mapsto (y,x)$, so $v = tu$ acts trivially; thus $(\C \otimes \Q[i]) \sharp H \cong (\C \otimes \Q[i]) \langle v\rangle \sharp_c \overline H$, a crossed product algebra in which the quotient group $\overline H = H/\langle v\rangle$ acts faithfully on the ordinary group ring $(\C \otimes \Q[i])\langle v\rangle$ of the cyclic group $\langle v\rangle$ over the coefficient ring $\C \otimes \Q[i]$, and $c$ is a 2-cocycle defining the extension $1\to\langle v\rangle \to H \to H/\langle v\rangle \to 1$; but this extension is split, so the cocycle $c$ can be taken to be trivial, and hence $(\C \otimes \Q [i]) \sharp H$ is an ordinary twisted group ring $(\C \otimes \Q [i]) \langle v\rangle \sharp \overline H$.  Moreover, $\overline H$ acts trivially on $\langle v\rangle$, as the extension is central, so the above reduces to $\C\langle v\rangle \otimes_{\Q} (\Q [i]\sharp \overline H)$. Now $\overline H$ acts faithfully on $\Q[i]$, so by Galois theory (e.g., \cite{DI}, Chapter III, Proposition 1.2), $\Q[i]\, \sharp\, \overline H\cong \End_{\Q} (\Q [i]) \cong \M_2(\Q)$, while $\C \langle v\rangle \cong \C \times \C$; thus the third factor of \eqref{3.1} decomposes as $\M_2(\C) \times \M_2(\C)$ and contributes two irreducible representations $W^+$ and $W^-$. Tracing through the isomorphisms, one sees that the actions of the generators of $G$ in these representations are as follows:
\begin{itemize}
\item On $W^+$, $s$ acts as $\bmat{
i &0\\
0 &-i
}$, while $t$ and $u$ act as
$\left[
\begin{array}{cc}
0 &1\\
1 &0
\end{array}
\right]$;
\item On $W^-$, $s$ acts as
$\left[
\begin{array}{cc}
i &0\\
0 &-i
\end{array}
\right]$, $t$ acts as
$\left[
\begin{array}{cc}
0 &1\\
1 &0
\end{array}
\right]$, and $u$ acts as
$\left[
\begin{array}{cc}
0 &-1\\
-1 &0
\end{array}
\right]$.
\end{itemize}

Finally, consider the fourth factor $(\C \otimes_{\Q} \Q [\zeta]) \sharp H \cong \C \otimes_{\Q} (\Q [\zeta ] \sharp H)$ in \eqref{3.1}. By Galois theory, $\Q[\zeta ] \sharp H \cong \End_{\Q}(\Q[\zeta ]) \cong \M_4(\Q)$, since $H = \Gal (\Q[\zeta]/\Q)$. Thus the fourth factor of the decomposition \eqref{3.1} is $\M_4(\C)$, and it contributes a 4-dimensional irreducible representation $X$. Using the basis $\{1,\zeta ,\zeta ^2,\zeta ^3\}$ for $\Q[\zeta ]$, one sees that:
\[ \text{$s$ acts by
$\left[
\begin{array}{cccc}
0 &0&0&-1\\
1 &0 &0&0\\
0 &1&0&0\\
0 &0&1&0
\end{array}
\right]$,\quad $t$ by $\left[
\begin{array}{cccc}
1 &0&0&0\\
0 &0 &0&-1\\
0 &0&-1&0\\
0 &-1&0&0
\end{array}
\right]$, \quad and $u$ by
$\left[
\begin{array}{cccc}
1 &0&0&0\\
0 &0 &0&1\\
0 &0&-1&0\\
0 &1&0&0
\end{array}
\right]$}.\]
This completes the determination of all the irreducible representations of $G$; the character table is given in Table 1.  As usual, the rows are indexed by the irreducible representations (up to isomorphism) of $G$ and the columns are indexed by (representatives of) the conjugacy classes of $G$.  The additional row at the bottom of the table records the character values of the induced representations $(\bbm{1}_{\Gamma _i})\negthickspace\uparrow^G_{\Gamma _i}$, which are easily computed via the formula for the character of an induced representation, or directly from the Schreier graphs \eqref{SchreierGraphs}.

\begin{center}
\begin{tabular}{c||c|c|c|c|c|c|c|c|c|c|c||}
{} & $1$ & $s^4$ & $s^2$ & $v$ & $s^2v$ & $s$ & $t$ & $u$ & $st$ & $su$ & $sv$ \\ \hline\hline
$\bbm{1}$ & $1$ & $1$ & $1$ & $1$ & $1$ & $1$ & $1$ & $1$ & $1$ & $1$ & $1$ \\ \hline
$\bbm1^{+-+}$ & 1 & 1 & 1 & $-1$ & $-1$ & 1 & $-1$ & 1 & $-1$ & 1 & $-1$ \\ \hline
$\bbm1^{++-}$ & 1 & 1 & 1 & $-1$ & $-1$ & 1 & 1 & $-1$ & 1 & $-1$ & $-1$ \\ \hline
$\bbm1^{+--}$ & 1 & 1 & 1 & 1 & 1 & 1 & $-1$ & $-1$ & $-1$ & $-1$ & 1\\ \hline
$\bbm1^-$ & 1 & 1 & 1 & 1 & 1 & $-1$ & 1 & 1 & $-1$ & $-1$ & $-1$ \\ \hline
$\bbm1^{--+}$ & 1 & 1 & 1 & $-1$ & $-1$ & $-1$ & $-1$ & 1 & 1 & $-1$ & 1 \\ \hline
$\bbm1^{-+-}$ & 1 & 1 & 1 & $-1$ & $-1$ & $-1$ & 1 & $-1$ & $-1$ & 1 & 1 \\ \hline
$\bbm1^{---}$ & 1 & 1 & 1 & 1 & 1 & $-1$ & $-1$ & $-1$ & 1 & 1 & $-1$ \\ \hline
$W^+$ & 2 & 2 & $-2$ & 2 & $-2$ & 0 & 0 & 0 & 0 & 0 & 0 \\ \hline
$W^-$ & 2 & 2 & $-2$ & $-2$ & $2$ & 0 & 0 & 0 & 0 & 0 & 0 \\ \hline
$X$ & 4 & $-4$ & 0 & 0 & 0 & 0 & 0 & 0 & 0 & 0 & 0 \\ \hline\hline
 $(\bbm{1}_{\Gamma _i})\negthickspace\uparrow^G_{\Gamma _i}$ & 8 & 0 & 0 & 4 & 0 & 0 & 2 & 2 & 0 & 0 & 0 \\ \hline\hline
\end{tabular}

\vspace{3pt}
{\it Table 1}
\end{center}

It is easy to compute the character of the induced representation $(\bbm{1}_{\Gamma _i})\negthickspace\uparrow^G_{\Gamma _i}$ in terms of the irreducible characters by orthonormal expansion, since the latter form an orthonormal basis for the space of class functions relative to the inner product $\langle\cdot,\cdot\rangle$ given by
\[\langle\chi,\psi
  \rangle=\frac1{|G|} \sum_{g\in G} \chi (g) \overline \psi (g).\]
One finds:

\begin{prop}\label{p3.2}
$\C[G/\Gamma _i] \cong \bbm{1} \oplus \bbm{1}^- \oplus W^+ \oplus X$.
\end{prop}

Using the Fourier inversion formula \cite{Se}
\[\varepsilon _V = \frac{\dim(V)}{|G|} \sum_{g\in G} \chi _V
(g^{-1})g\]
for the primitive central idempotent $\varepsilon _V$ associated to an irreducible representation $V$ of $G$ one can easily determine bases for the irreducible constituents of the representation $\C[G/\Gamma_i]$ in terms of the bases of cosets. Consider first $\C[G/\Gamma_1]$. Let $u_i$ denote the coset $s^i\Gamma _1$, $0\le i\le 7$.  Let $e_1 = (\varepsilon _{1^+})\cdotp u_0 = \frac1 8 \sum^7_{i=0}u_i$, $e_2 = (\varepsilon_{1^-})\cdotp u_0 = \frac1 8 \sum^7_{i=0}(-1)^iu_i$, $e_3 = (\varepsilon_{W^+})\cdotp u_0 =\frac1 4(u_0-u_2+u_4-u_6)$, $e_4 = (\varepsilon _{W^+})\cdotp u_1 =\frac14(u_1-u_3+u_5-u_7)$, $e_5 = \varepsilon _X\cdotp u_0 =\frac12(u_0-u_4)$, $e_6 = \varepsilon _X\cdotp u_1 =\frac12(u_1-u_5)$, $e_7 = \varepsilon _X\cdotp u_2 =\frac12(u_2-u_6)$, and $e_8 = \varepsilon _X\cdotp u_3 =\frac12(u_3-u_7)$. Then the following is immediate:

\begin{prop}\label{BasesOfIrreps1} Let $e_1,\dotsc,e_8$ be as defined above.  Then:
\begin{itemize}
\item $\{e_1\}$ is a basis of the $\bbm1^+$ summand of $\C[G/\Gamma _1]$.

\item $\{e_2\}$ is a basis of the $\bbm1^-$ summand of $\C[G/\Gamma _1]$.

\item $\{e_3,e_4\}$ is a basis of the $W^+$ summand of $\C[G/\Gamma _1]$.

\item $\{e_5,e_6,e_7,e_8\}$ is a basis of the $X$ summand of $\C[G/\Gamma _1]$.
\end{itemize}
\end{prop}

Similarly, turning to $G/\Gamma _2$, let $v_i$ denote the coset $s^i\Gamma _2$, $0\le i\le 7$.
\begin{prop}\label{BasesOfIrreps2}
Let $f_1,f_2\ldots f_7$ be defined by the same formulas defining the $e_i$, but with each $u_i$ replaced by $v_i$. Then:
\begin{itemize}
\item $\{f_1\}$ is a basis of the $\bbm1^+$ summand of $\C[G/\Gamma _2]$.

\item $\{f_2\}$ is a basis of the $\bbm1^-$ summand of $\C[G/\Gamma _2]$.

\item $\{f_3,f_4\}$ is a basis of the $W^+$ summand of $\C[G/\Gamma _2]$.

\item $\{f_5,f_6,f_7,f_8\}$ is a basis of the $X$ summand of $\C[G/\Gamma _2]$.
\end{itemize}
\end{prop}

We now make explicit the most general equivalence of the permutation representations $\C[G/\Gamma _1]$ and $\C[G/\Gamma _2]$; since these representations are both equivalent to $\bbm1\oplus \bbm1^- \oplus W^+ \oplus X$, this reduces to determining all the intertwining isomorphisms of the latter.

\begin{remarks}\label{MapsOfSubreps}
From the above, the following assertions are immediate:
\begin{enumerate}
\item\label{1+->1+}
For any nonzero scalar $a$, the map
    sending $e_1\mapsto af_1$ is a $G$-isomorphism of the $\bbm1^+$
    summand of $\C[G/\Gamma _1]$ onto the $\bbm1^+$ summand of
    $\C[G/\Gamma _2]$.

\item For any nonzero scalar $b$, the map
    sending $e_2\mapsto bf_2$ is a $G$-isomorphism of the $\bbm1^-$
    summand of $\C[G/\Gamma _1]$ onto the $\bbm1^-$ summand of
    $\C[G/\Gamma _2]$.

\item The actions of $s$, $t$ and $u$ on $e_3$ and $e_4$ coincide with their actions on $f_3$ and $f_4$; thus for any nonzero $c$, the map sending $e_3 \mapsto cf_3$ and $e_4 \mapsto cf_4$ is a $G$-isomorphism of the $W^+$ summand of $\C[G/\Gamma _1]$ onto the $W^+$ summand of $\C[G/\Gamma _2]$. Moreover, by Schur's Lemma, this is the most general such isomorphism.

\item One easily computes that the most general $G$-isomorphism of $\span\{e_5,e_6,e_7,e_8\}$ onto $\span\{f_5,f_6,f_7,f_8\}$ is given by
\[\left[
\begin{array}{cccc}
0 &d&0&-d\\
d &0 &d&0\\
0 &d&0&d\\
-d &0&d&0
\end{array}
\right],\]
where $d\ne 0$. Thus, letting $h_5 = f_6-f_8$, $h_6=f_5+f_7$, $h_7=f_6+f_8$, $h_8=-f_5+f_7$, we see that, relative to the bases $\{e_5,e_6,e_7,e_8\}$ and $\{h_5,h_6,h_7,h_8\}$, the most general isomorphism of the $X$ summand of $\C[G/\Gamma _1]$ onto the $X$ summand of $\C[G/\Gamma _2]$ is given by $e_i \mapsto dh_i$, $i=5,6,7,8$.

\item\label{GeneralIntertwiningMap} Finally, let $h_i=f_i$ for $1\le i\le 4$. Then the most general $G$-isomorphism $\Phi : \C[G/\Gamma _1] \to \C[G/\Gamma _2]$ is given, relative to the bases $\{e_1,e_2,\ldots,e_8\}$ and $\{h_1,h_2,\ldots,h_8\}$, by the matrix  diag$(a,b,c,c,d,d,d,d)$, where $abcd\ne0$.  Changing bases, we see that relative to the natural coset bases $\{u_0,u_1\ldots u_7\}$ and $\{v_0,v_1\ldots v_7\}$, the general intertwining isomorphism $\Phi$ is given by the matrix
\begin{equation}\label{TransplantationMatrix}
A = \left[
\begin{array}{cccccccc}
\alpha  &\beta  &\gamma  &\delta  &\alpha  &\delta  &\gamma  &\beta \\
\beta  &\alpha  &\beta  &\gamma  &\delta  &\alpha  &\delta  &\gamma \\
\gamma  &\beta  &\alpha  &\beta  &\gamma  &\delta  &\alpha  &\delta \\
\delta  &\gamma  &\beta  &\alpha  &\beta  &\gamma  &\delta  &\alpha \\
\alpha  &\delta  &\gamma  &\beta  &\alpha  &\beta  &\gamma  &\delta \\
\delta  &\alpha  &\delta  &\gamma  &\beta  &\alpha  &\beta  &\gamma \\
\gamma  &\delta  &\alpha  &\delta  &\gamma  &\beta  &\alpha  &\beta \\
\beta  &\gamma  &\delta  &\alpha  &\delta  &\gamma  &\beta  &\alpha
\end{array}
\right],
\end{equation}
where $\alpha  = \frac18(a+b+2c)$, $\beta  = \frac18(a-b+4d)$, $\gamma =\frac18(a+b-2c)$, $\delta =\frac18(a-b-4d)$ and $\alpha $, $\beta $, $\delta $ are chosen so that $abcd \ne 0$.
\end{enumerate}
\end{remarks}

\begin{remarks} The above conclusions can also be reached by more elementary computations:
\begin{enumerate}
\item The transplantation matrix $A$ relative to the bases given by the cosets can be computed in a na\"ive way simply by determining the conditions on the entries of $A$ that are forced by the requirement that $A$  intertwine the two permutation representations.

\item The decomposition of the representations $\C[G/\Gamma _i]$ into irreducible representations can also be determined by considering the equation $s^8 -1 =0$ and looking at the possible degrees of irreducible constituents. From these irreducible decompositions, one can recover Buser's transplantation of eigenfunctions on the isospectral Riemann surfaces described in \cite{Bu2}.

\item For a systematic study of transplantation with many interesting examples, see \cite{Herbrich}.
\end{enumerate}
\end{remarks}

A simple choice of $A$ for our purposes is given by setting $c=d=2$, $a=6$, $b=-2$, so $\alpha =1$, $\beta =2$, $\gamma =\delta =0$. We can now describe explicitly how to transplant a Neumann eigenfunction from $M_1$ to $M_2$. Recall that $M_i$ is constructed by gluing together copies of a fundamental tile $\ms{T}$, the copies labeled by the cosets $G/\Gamma _i$. Let $F$ be a $\lambda $-eigenfunction on $M_1$; let $F_i$ denote its restriction to the tile labeled by the coset $s^i\Gamma _1$. Then $F$ can be represented by the vector of functions on $\ms{T}$ given by $[F_0~F_1~F_2\cdots F_7]$. (Note that there is no ambiguity about what it means to regard an $F_i$ as a function on $\ms{T}$; there is a unique isometry of $\ms{T}$ with each tile of $M_1$ or $M_2$ that preserves the labeling of the boundary edges.) Now let $H$ denote the function on $M_2$ whose restriction $H_i$ to the tile of $M_2$ labeled $s^i\Gamma _2$ is given by the matrix product
\[\bmat{H_0&H_1&H_2&\dotsc&H_7}= \bmat{F_0&F_1&F_2&\dotsc &F_7}A,\]
where $A$ is the intertwining matrix defined above. If $H$ is smooth, then it is certainly an eigenfunction, since this is a local condition; thus checking Neumann isospectrality reduces to checking that:
\begin{enumerate}
\item The functions $H_i$ fit together smoothly across the interfaces
between tiles;

\item The function $H$ satisfies Neumann boundary conditions.
\end{enumerate}

These assertions are easily checked by inspection of the paper models or by looking at the graphs \eqref{SchreierGraphs}; we illustrate briefly.  Consider the interface between tiles 0 and 4 of $M_2$; it is depicted in the figure below.
\begin{center}
\includegraphics[width=3.5in]{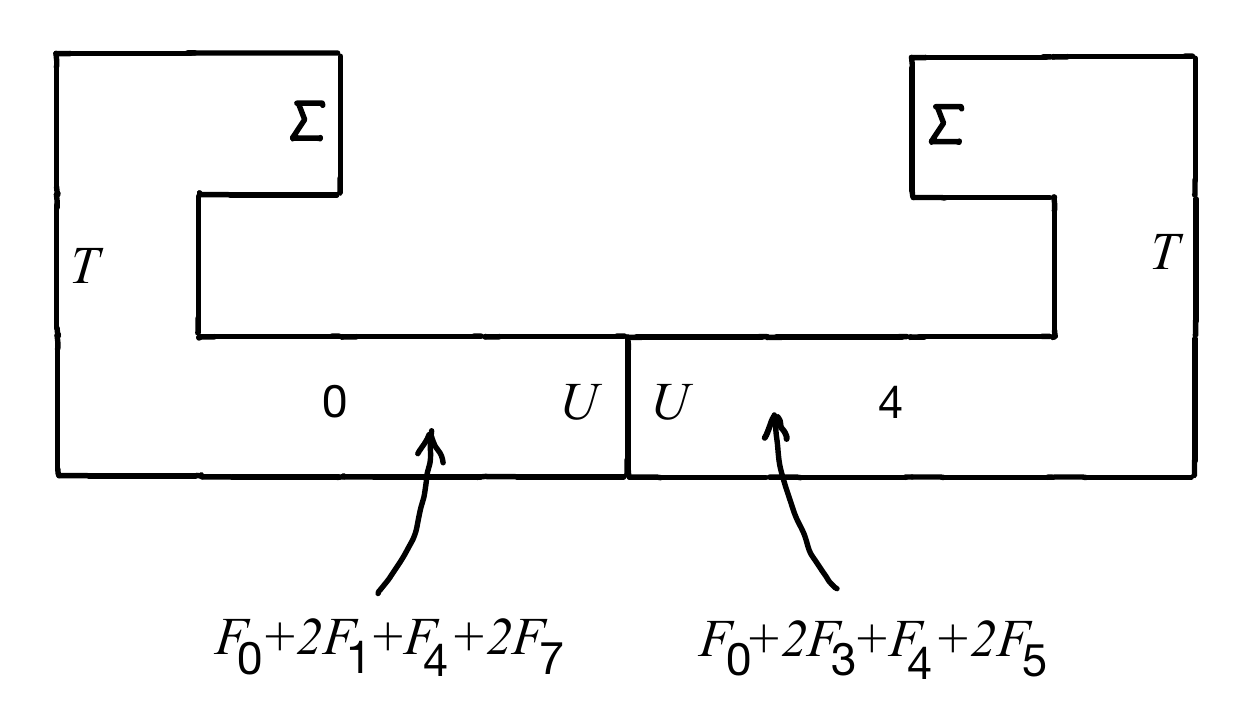}
\end{center}
We will show that $H_0$ and $H_4$ fit together smoothly across this $U$-edge. It suffices to show that the following condition is satisfied:
\begin{equation}\label{SmoothnessCondition}
\begin{split}
\text{$H_0$ and $H_4$ coincide along their common $U$-edge, and their}\\\text{ inner normal derivatives there are negatives of each other.}
\end{split}
\end{equation}

\noindent Indeed, if \eqref{SmoothnessCondition} holds, then by Green's formula the function $H$ (which equals $H_0$ on tile 0 and $H_4$ on tile 4) is a local weak solution of $\Delta u = \lambda u$, so by regularity of local weak solutions, $H$ is smooth along the $U$-edge. We now verify that \eqref{SmoothnessCondition} holds. The functions $H_0$ and $H_4$ are given by $H_0=F_0+2F_1+F_4+2F_7$, $H_4 = F_0 +2F_3+F_4 +2F_5$. Inspection of Figure~4 or of the graphs \eqref{SchreierGraphs} shows that in $M_1$, tiles 1 and 3 meet along a $U$-edge, so $2F_1$ and $2F_3$ and their derivatives fit together smoothly  along their common $U$-edge; similarly, $2F_7$ and $2F_5$ fit together. As for $F_0$ and $F_4$, they have vanishing normal derivatives on the edge. Thus $H_0$ and $H_4$ fit together smoothly across the $U$-interface. To see that $H$ satisfies Neumann boundary conditions on $M_2$, consider for example the tile 0; we will show that $H_0$ satisfies the Neumann condition on the $T$-edge. First, the $T$-edges of tiles 0 and 4 of $M_1$ are boundary edges, so $F_0$ and $F_4$ have vanishing normal derivative on the $T$-edge. Figure~4 and the graphs \eqref{SchreierGraphs} also show that tiles 1 and 7 of $M_1$ share a $T$-edge; since the (unique) isometry used to identify tiles 1 and 7 of $M_1$ is a reflection in their common $T$-edge, it is clear that on the $T$-edge of tile 0 in $M_2$, the normal derivatives of $F_1$ and $F_7$ are negatives of each other. Thus the normal derivative of $H_0 = F_0 +2F_1+F_4+2F_7$ vanishes on the $T$-edge, as desired.  One performs a similar verification on each tile of $M_2$.

Thus any Neumann eigenfunction on $M_1$ can be transplanted to a Neumann eigenfunction on $M_2$; using the inverse of the matrix $A$, one can transplant eigenfunctions from $M_2$ to $M_1$. Thus $M_1$ and $M_2$ are Neumann isospectral.

\section{A folklore argument of Fefferman}\label{Fefferman}

In \cite{GWW}, planar domains that are isospectral for either Neumann or Dirichlet boundary conditions were exhibited. The domains are underlying spaces of orbifolds (with boundary) $\cO_1$ and $\cO_2$ whose singular sets consist of disjoint mirror arcs, and the Neumann isospectrality was established by showing that $\cO_1$ and $\cO_2$ are isospectral as orbifolds. The Dirichlet isospectrality follows because there are  surfaces (with boundary) $S_1$ and $S_2$ that can be viewed as the orientation double covers of the orbifolds $\cO_1$ and $\cO_2$ (obtained by ``doubling'' the orbifolds along the mirror edges); the surfaces $S_1$ and $S_2$ are themselves isospectral by Sunada's Theorem; these examples were discovered by Buser \cite{Bu1}. In fact, for each of $S_1$ and $S_2$, the reflection symmetry decomposes the space of smooth functions on $S_i$ as the direct sum of a $(+1)$-eigenspace (the reflection-invariant functions) and a $(-1)$-eigenspace (the reflection-anti-invariant functions); the latter are the functions satisfying the Dirichlet boundary condition, and using this decomposition, one easily deduces Dirichlet isospectrality from the isospectrality of $S_1$ and $S_2$ and of $\cO_1$ and $\cO_2$. Thus $|\cO_1|$ and $|\cO_2|$ are Dirichlet isospectral domains. One can also make explicit a Dirichlet transplantation matrix (see \cite{Be3}).

As was observed by Peter Doyle in connection with some of the examples in \cite{BCDS}, if the orientation double covers are not isospectral, then the above argument fails, so there is no reason to expect $M_1$ and $M_2$ to be Dirichlet isospectral.  We will show in fact that our $M_1$ and $M_2$ have a different lowest Dirichlet eigenvalue, at least when $M_1$ and $M_2$ are constructed using a more symmetrical fundamental tile than that used in section~\ref{Construction}. The extra symmetry will permit an adaptation of an argument due to C.~Fefferman showing that the two planar domains $S$ and $C$ shown below are not Dirichlet isospectral.
\begin{center}
\includegraphics[width=4.5in]{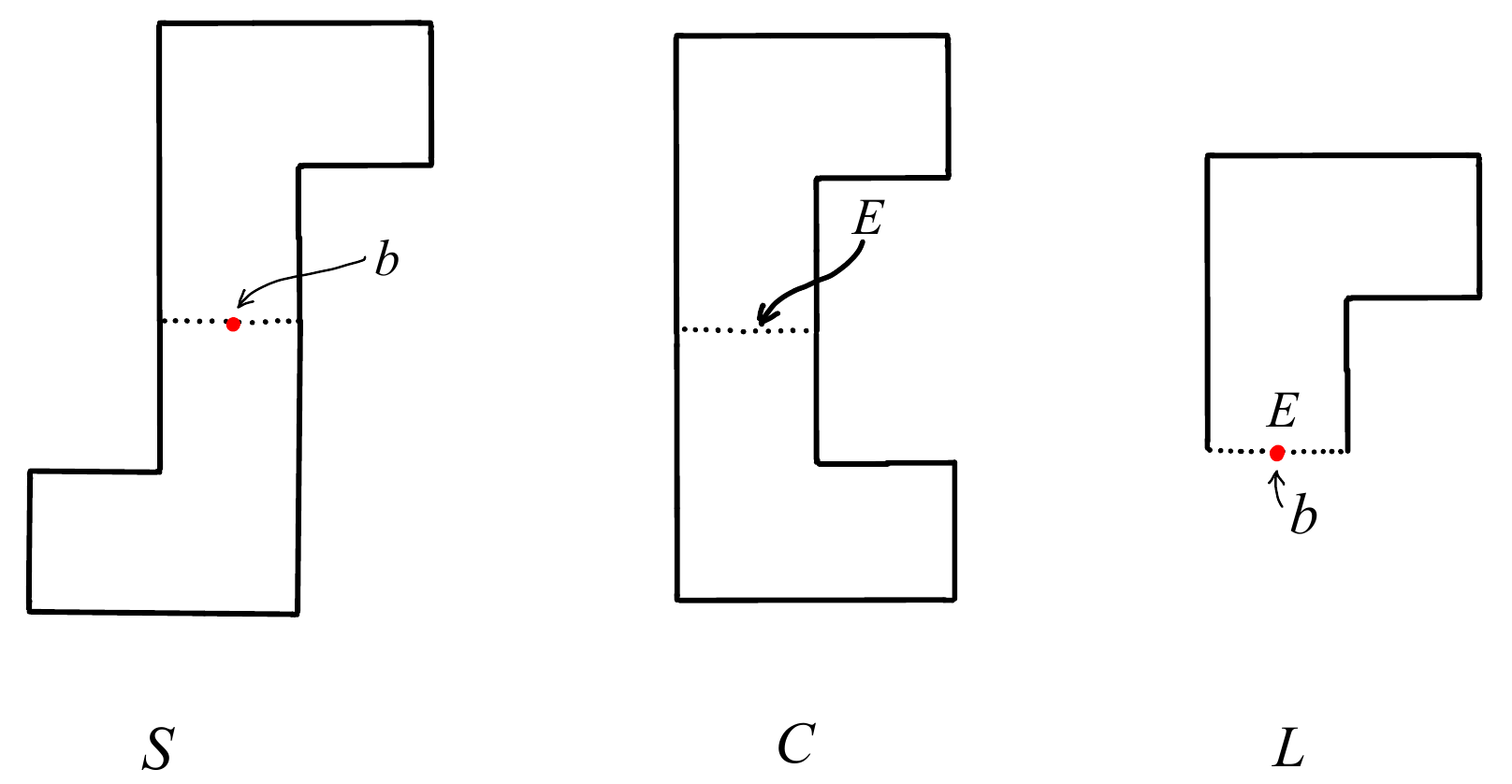}
\end{center}
The tile $L$ on the right above is the upper half of both the domains $S$ and $C$.

Fefferman's argument exploits the fact that the domain $C$ has a reflection symmetry (reflection through the dotted line $E$), while the domain $S$ has an involutive rotation symmetry (rotation by $\pi$ around the barycenter $b$). We sketch Fefferman's argument below.

For any Riemannian manifold (with boundary) $M$ and $p$-forms $\alpha$, $\beta$ on $M$, write $(\cdot,\cdot)_M$ for the $L^2$ inner product: $(\alpha,\beta)_M= \int_M \langle \alpha,\beta \rangle\,dvol $, where $\langle\cdot,\cdot\rangle$ is the pointwise inner product on $p$-covectors and $dvol$ is the Riemannian measure. We recall (see \cite{Be1} or \cite{C}) that, for a Riemannian manifold $M$ with boundary, the lowest nonzero Neumann eigenvalue $\mu$ is given by the infimum of the Rayleigh quotients:
\begin{equation}\label{NeumannRayleighQuotient}
\mu=\inf_{f\in H^1,\,f\perp1}R_M(f) = \inf_{f\in H^1,\,f\perp1} \frac{(df,df)_M}{(f,f)_M},
\end{equation}
where $f$ ranges over the space $C^{\infty}(M)$ of smooth functions on $M$ (or equivalently, over its $H^1$-completion, the Sobolev space $H^1$ of functions having one distributional derivative in $L^2$) that are $L^2$-orthogonal to the constant functions. The lowest Dirichlet eigenvalue $\lambda $ is given by
\[\lambda=\inf_{f\in H^1_0}R_M(f)=\inf_{f\in H_0^1}\frac{(df,df)_M}{(f,f)_M},\]
the infimum of the Rayleigh quotients with $f$ ranging over the space of smooth functions with compact support in the interior of $M$ (or equivalently, over its $H^1$-completion $H^1_0$). Finally, for mixed boundary conditions (i.e., Neumann boundary conditions on an open submanifold $N$ of $\partial M$, but Dirichlet conditions on an open submanifold $D\subseteq \partial M$ such that $\partial M=\overline{N\cup D}$), one allows $f$ to range over the space of smooth functions on $M$ supported away from $D$ (or over its Sobolev completion $H^1_{\operatorname{mixed}}$). In each case, a function $f$ in the pertinent Sobolev completion whose Rayleigh quotient realizes the infimum is an eigenfunction for the lowest nonzero eigenvalue.
\begin{thm}(Fefferman)
The domains $S$ and $C$ depicted above are not Dirichlet isospectral.  In fact, if $\lambda_{\Omega}$ denotes the lowest Dirichlet eigenvalue of a domain $\Omega\subseteq\R^2$, then $\lambda_C<\lambda_S$.
\end{thm}

\begin{proof} Fefferman's argument runs as follows.  Let $f_C$ be a normalized (i.e., of unit $L^2$-norm) eigenfunction on $C$ for the lowest Dirichlet eigenvalue $\lambda_C$ of $C$.  By Courant's nodal domain theorem (see \cite{C}), $f_C$ is never zero on the interior $C^{\circ}$ of $C$, so without loss of generality, assume that $f_C>0$ on $C^{\circ}$.  The eigenspace associated with the lowest eigenvalue is one-dimensional, since no two functions that are strictly positive on $C^{\circ}$ can be $L^2$-orthogonal.  Let $\varrho$ be the obvious reflection symmetry of $C$ (reflection through the line $E$ in the illustration above).  Then $f_C\circ\varrho$ is another normalized $\lambda_C$-eigenfunction on $C$, so $f_C\circ\varrho=\pm f_C$.  But $f_C\ge0$, so $f_C\circ\varrho\ge0$, and thus $f_C\circ\varrho=f_C$.   Thus $f_C$ is invariant under the reflection in the segment $E$, so restricting $f_C$ to the upper half $L$ of $C$ yields a function $f_L:=(f_C)\vert_L$ that is zero on the three edges of $L$ other than $E$, but has zero normal derivative on $E$. Also, $f_L$ is strictly positive on the interior of $L$.  Thus $f_L$ is an eigenfunction on $L$ corresponding to the lowest eigenvalue for the following mixed problem: Neumann boundary conditions on the edge $E$ and Dirichlet boundary conditions on the other three edges of $L$.  Thus by the variational characterization of eigenvalues (see \cite{Be1}), $(f_C)\vert_L$ realizes the infimum of the Raleigh quotients
\[\inf_{f\in H^1_{\operatorname{mixed}}}R_L(f)=\inf_{f\in H^1_{\operatorname{mixed}}}\frac{(df,df)_{L}}{(f,f)_{L}},\]
where $H^1_{\operatorname{mixed}}$ is the Sobolev space for the mixed problem, the $H^1$-completion of the space of smooth functions supported away from the three edges of $L$ other than $E$.  Because of the reflection invariance of $f_C$,
\[\frac{(df_C,df_C)_C}{(f_C,f_C)_C}=2\frac{(df_C,df_C)_L}{(f_C,f_C)_L}.\]
Now consider a normalized eigenfunction $f_S$ for the lowest Dirichlet eigenvalue $\lambda_S$ of $S$.  By an analogous argument to that above, $f_S$ satisfies $f_S=f_S\circ\sigma$, where $\sigma$ is the involutive rotation symmetry of $S$.  It follows that
\[\frac{(df_S,df_S)_S}{(f_S,f_S)_S}=2\frac{(df_S,df_S)_L}{(f_S,f_S)_L}.\]

The restriction $(f_S)\vert_L$ to $L$ vanishes on the three edges of $L$ other than $E$, i.e., $(f_S)\vert_L\in H^1_{\operatorname{mixed}}$.  Since $(f_C)\vert_L$ realizes the infimum and $f_S\in H^1_{\operatorname{mixed}}$, it follows that
\begin{equation}\label{RayleighQuotIneq}
\frac{(df_C,df_C)_L}{(f_C,f_C)_L}\le\frac{(df_S,df_S)_L}{(f_S,f_S)_L}.
\end{equation}
If equality holds, then $(f_S)\vert_L$ is also a nonnegative normalized eigenfunction for the mixed problem, so $(f_S)\vert_L=(f_C)\vert_L$.  Since eigenfunctions are real analytic, it follows that $f_C $ and $f_S$ must agree upon their common domain of definition when the top halves of $S$ and $C$ are superimposed, i.e., on the domain $\Gamma$ depicted below:
\begin{center}
\includegraphics[width=1.25in]{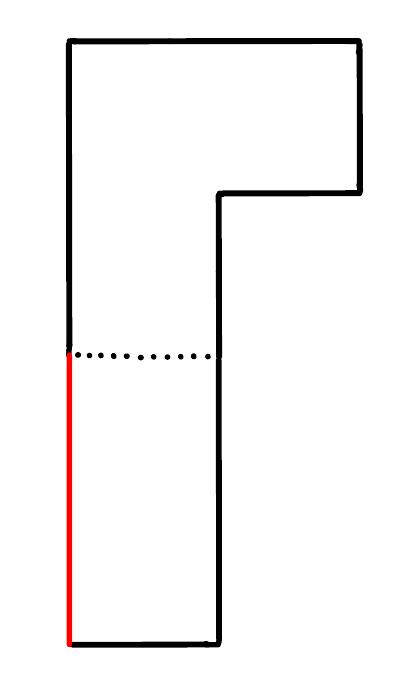}
\end{center}
But $f_C$, as a Dirichlet eigenfunction on $C$, must vanish on the red edge of $\Gamma$, while $f_S$ cannot vanish there, since it is strictly positive on the interior of $S$.  This contradiction implies that the inequality \eqref{RayleighQuotIneq} is strict.  Thus
\[\lambda_C=\frac{(df_C,df_C)_C}{(f_C,f_C)_C}=\frac{2(df_C,df_C)_L}{2(f_C,f_C)_L}<\frac{2(df_S,df_S)_L}{2(f_S,f_S)_L}=\frac{(df_S,df_S)_S}{(f_S,f_S)_S}=\lambda_S,\]
so the lowest eigenvalue of $C$ is smaller than the lowest eigenvalue of $S$.
\end{proof}

\section{Neumann isospectral but not Dirichlet isospectral surfaces with boundary}\label{NeumannNotDirichlet}
We turn now to an adaptation of Fefferman's argument to our setting.  The bordered surfaces $M_1$ and $M_2$ are the underlying spaces of orbifolds $\cO_1$ and $\cO_2$ constructed exactly as in section~\ref{Construction}, but using a $Y$-shaped fundamental tile: the underlying space $Y:=\abs{\cO_0}$ of the orbifold $\cO_0$ depicted below rather than the tile of Figure~5.
\begin{center}
\includegraphics[width=1.85in]{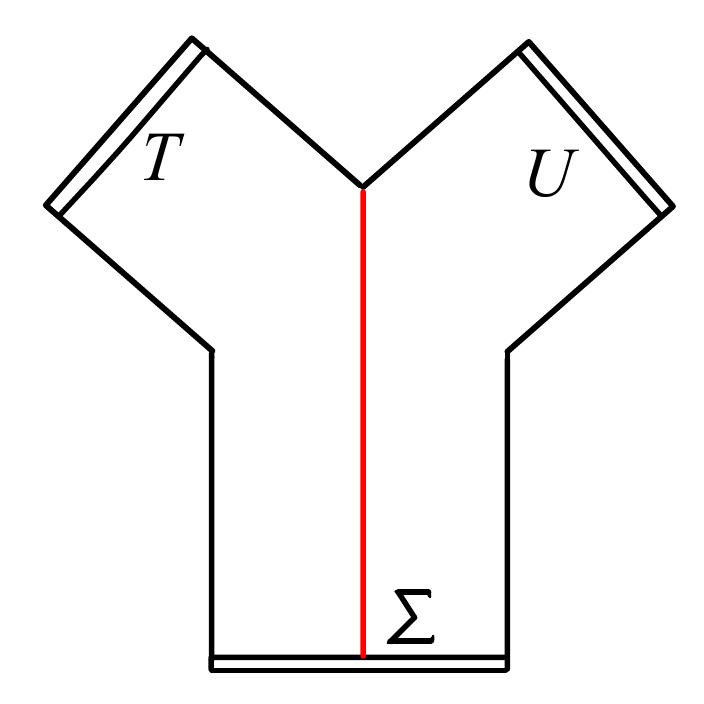}
\end{center}
This tile has a symmetry that will be exploited in what follows.

\noindent The manifolds $M_1$ and $M_2$ thus obtained are depicted in Figure~6; as before, both an embedded and an immersed version of $M_1$ are shown.

\begin{center}
\includegraphics[width=6.75in]{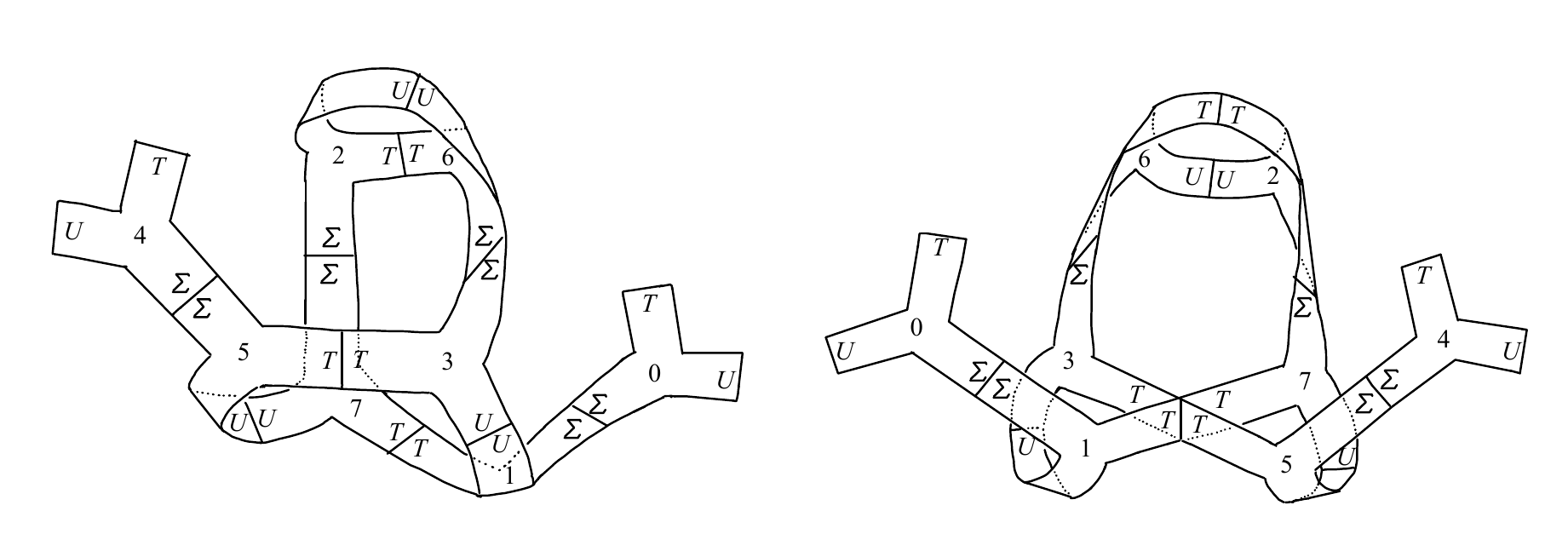}
$\qquad\qquad M_1=\abs{\cO_1}\qquad\qquad\qquad\qquad\qquad\qquad\qquad\qquad\qquad M_1=\abs{\cO_1}$\linebreak
\includegraphics[width=3.0in]{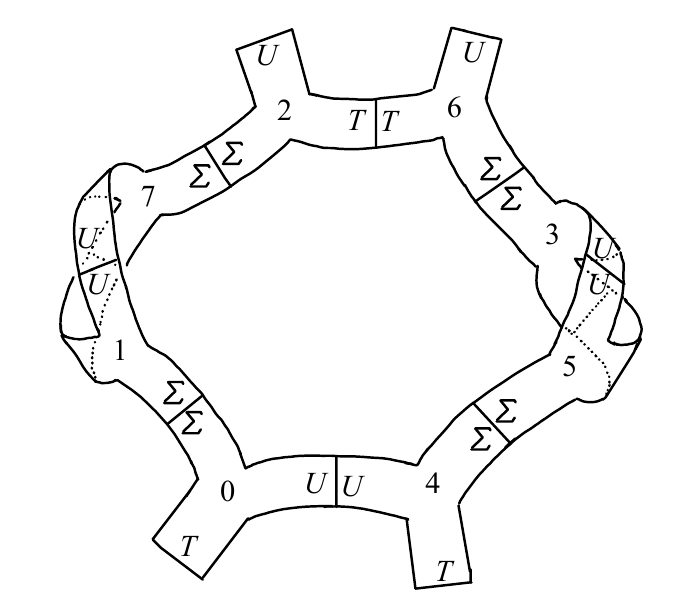}\linebreak
$M_2=\abs{\cO_2}$\linebreak\linebreak
{\it Figure 6}
\end{center}

\begin{theorem}\label{t2}
The bordered surfaces $M_1$ and $M_2$ depicted in Figure~6 are isospectral with Neumann boundary conditions, but not with Dirichlet boundary conditions.
\end{theorem}

\begin{proof} The Neumann isospectrality was established in section~\ref{Construction}; it remains only to show that Dirichlet isospectrality fails. Suppose then that $M_1$ and $M_2$ are Dirichlet isospectral.

Note that $M_1$ has an involutive isometry $\rho_1$ whose effect is to interchange tiles $i$ and $i-4$ (mod 8) (this involutive symmetry is best visualized in the immersed picture in Figure~6). The fixed-point set of $\rho_1$ consists of the common $T$-edges and $U$-edges of tiles 2 and 6. (Although the common $T$-edge of tiles 1 and 7 and the common $T$-edge of tiles 3 and 5 coincide in the picture of $M_1$ immersed, they are of course distinct and are interchanged by $\rho_1$. Thus $\rho_1$ is not quite a ``reflection in a plane perpendicular to the paper'' as the picture would suggest.) The quotient orbifold $Q_1 = M_1/\langle \rho_1\rangle$ is shown in Figure~7. Also, $M_2$ has a reflection symmetry $\rho_2$ whose effect is to interchange tiles $i$ and $i-4$ (mod 8); the fixed point set consists of the common $T$-edge of the tiles 2 and 6 and the common $U$-edge of tiles 0 and 4. The quotient orbifold $Q_2 = M_2/\langle \rho_2\rangle$ is also shown in Figure~7. Note that $Q_1$ and $Q_2$ have the same underlying space $S = |Q_1| = |Q_2|$.

\begin{center}
\includegraphics[width=6.25in]{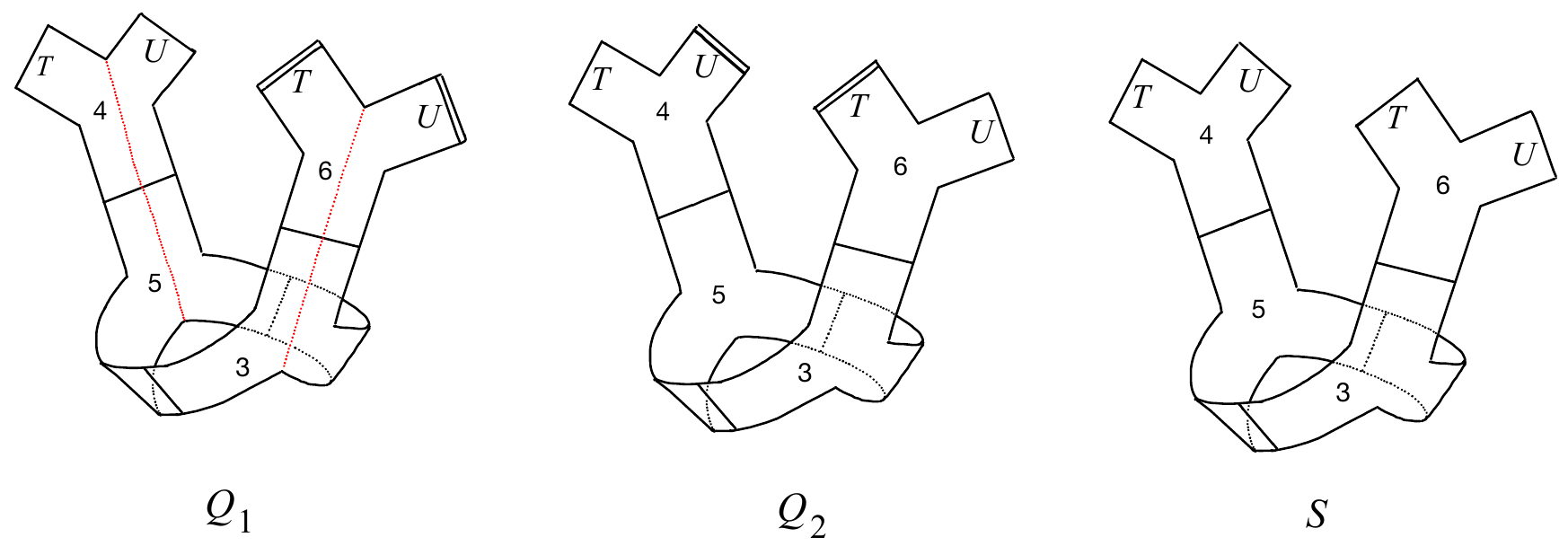}
{\it Figure 7}
\end{center}

Consider a Dirichlet eigenfunction $\varphi$ for the lowest Dirichlet eigenvalue $\lambda$ of $M_1$. By the Courant nodal domain theorem \cite{C}, by replacing $\varphi$ by $-\varphi$ if necessary, we can assume that $\varphi$ assumes only positive values on the interior of $M_1$. Thus the lowest eigenvalue $\lambda $ has multiplicity one, since no two everywhere-positive functions could be orthogonal. But $\rho_1$ is an isometry, so $\varphi  \circ \rho_1$ is also a $\lambda $-eigenfunction; hence $\varphi\circ \rho_1 = \varphi$, i.e., $\varphi$ is $\rho_1$-invariant. Thus $\varphi$ is a Dirichlet eigenfunction on the orbifold $Q_1$, that is, an eigenfunction on the underlying surface $S = |Q_1|$ in Figure~7 for the following  mixed boundary conditions:

\begin{equation}\label{NeumannOn6Tand6U}
\begin{cases}
\text{Neumann conditions on the boundary edges $6T$ and $6U$,\quad}\\ \text{Dirichlet boundary conditions on all other boundary edges.}
\end{cases}
\end{equation}

\noindent Moreover, $\varphi$ must be a lowest eigenfunction on $S$ for this mixed boundary value problem, since it is everywhere positive on the interior. Thus $\lambda$ is the lowest eigenvalue on $S$ for the mixed problem \eqref{NeumannOn6Tand6U}.

Since $M_1$ and $M_2$ are assumed Dirichlet isospectral, $\lambda$ is also the lowest Dirichlet eigenvalue of $M_2$. Let $\psi$ be a $\lambda$-eigenfunction on $M_2$. By the same argument, $\psi$ is $\rho_2$-invariant, so is a Dirichlet eigenfunction on the orbifold $Q_2$, i.e., an eigenfunction on the underlying space $S$ for the following mixed boundary conditions:

\begin{equation}\label{NeumannOn6Tand4U}
\begin{cases}\text{Neumann conditions on the boundary edges $6T$ and $4U$,}\\ \text{Dirichlet boundary conditions on all other boundary edges.}
\end{cases}
\end{equation}

\noindent Thus $\lambda $ is also the lowest eigenvalue on $S$ for the problem \eqref{NeumannOn6Tand4U}.  We conclude that the mixed eigenvalue problems \eqref{NeumannOn6Tand6U} and \eqref{NeumannOn6Tand4U} on $S$ have the same lowest eigenvalue $\lambda $.

Now the orbifold $Q_1$ itself has an involutive symmetry $\tau$ (reflection in the red line in the drawing of $Q_1$ in Figure~7), and the quotient orbifold $Q = Q_1/\langle \tau \rangle$ is shown in Figure~8 together with its underlying space $\Omega  = |Q|$. The surface $S$ can be recovered by doubling the planar domain $\Omega $ along the boundary edges $C$ and $E$.
\begin{center}
\includegraphics[width=5.5in]{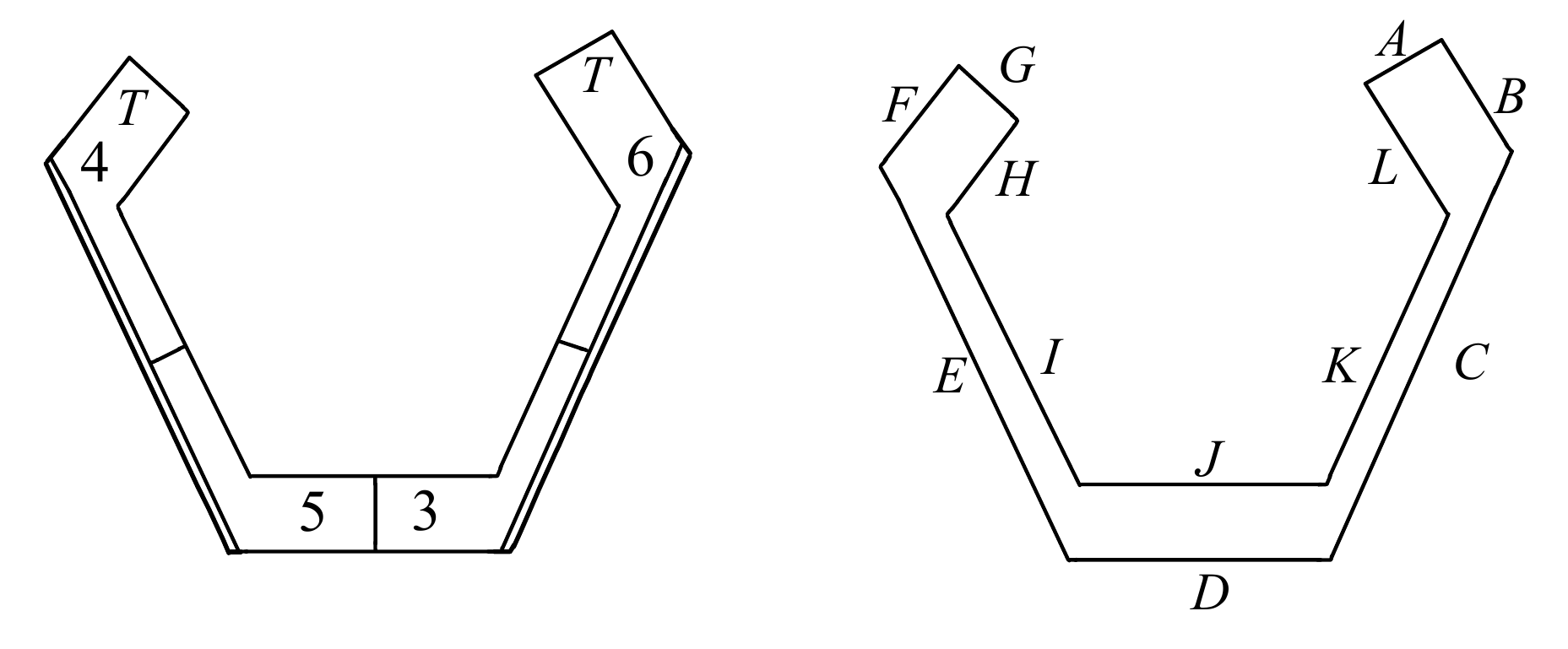}
$Q\qquad\qquad\qquad\qquad\qquad\qquad\qquad\qquad\qquad\qquad\qquad\qquad\Omega$\linebreak
\includegraphics[width=2.5in]{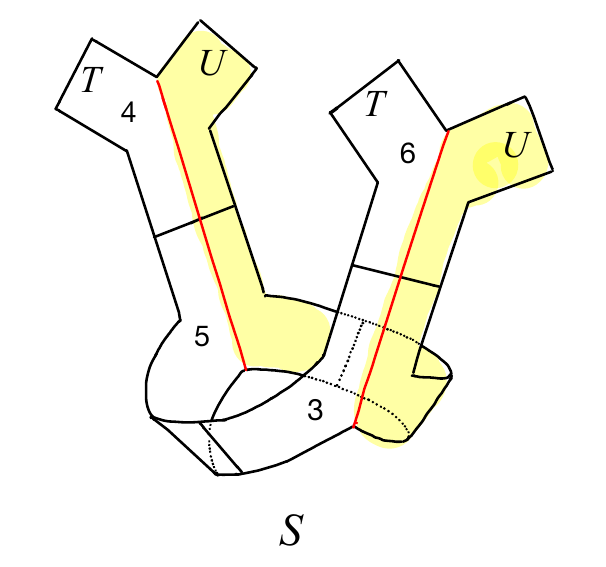}\linebreak
{\it Figure 8}
\end{center}

As before, the lowest eigenfunction $\varphi$ on $Q_1$ must be $\tau$-invariant, so it must be a lowest eigenfunction on the quotient orbifold $Q$, i.e., a lowest eigenfunction on $\Omega  = |Q|$ for the following boundary conditions:

\begin{equation}\label{Dagger}
\begin{cases}
\text{Neumann boundary conditions on edges $A$, $C$  and $E$,}\\
\text{Dirichlet conditions on all other boundary edges.\quad}
\end{cases}
\end{equation}

\noindent Thus the lowest eigenvalue of $\Omega $ for the boundary condition \eqref{Dagger} is also $\lambda$.

Now the lowest eigenfunction $\varphi$ on $\Omega $ for the mixed conditions \eqref{Dagger} achieves the infimum of the Rayleigh quotients
\[\inf_{f\in H^1_{\operatorname{mixed}}}  \frac{(df,df)_\Omega}{(f,f)_\Omega},\]
where $f$ ranges over the $H^1$-completion $H^1_{\operatorname{mixed}}$ of the space of smooth functions on $\Omega$ supported away from all boundary edges except $A$, $C$ and $E$.

Now let $\psi$ be a lowest Dirichlet eigenfunction on $Q_2$. Consider its restriction to the unshaded copy of $\Omega$ in Figure~8 which forms the left half of the surface $S$. Figure~7 shows that $\psi$ satisfies Dirichlet conditions on all boundary edges except $A$, $C$ and $E$, since all other edges lie in the orbifold boundary $\partial Q_2$; thus $\psi \in H^1_{\operatorname{mixed}}$. Now $\psi$ is an eigenfunction on $S$ for the problem~\eqref{NeumannOn6Tand4U}, so $\frac{(d\psi,d\psi)_S}{(\psi,\psi)_S} = \lambda$. Also, note that the orbifold $Q_2$ has an involutive ``rotation by $\pi$'' symmetry, and $\psi $ must be invariant under this symmetry. Thus
\[\frac{(d\psi ,d\psi )_\Omega }{(\psi,\psi)_\Omega}= \frac{\frac12(d\psi,d\psi)_S}{\frac12(\psi,\psi)_S }=\frac{(d\psi ,d\psi )_S}{(\psi,\psi)_S} = \lambda.\]
Thus $\psi$ realizes the infimum
\[\lambda = \inf_{f\in H^1_{\operatorname{mixed}}}\frac{(df,df)_\Omega}{(f,f)_\Omega},\]
so $\psi$ is a lowest eigenfunction on $\Omega$ for the boundary conditions~\eqref{Dagger}. Since the multiplicity of the lowest eigenvalue is one, it follows that $\psi$ and $\varphi$ agree on $\Omega $ (after multiplying $\psi$ by a nonzero scalar).  Since $\psi$ and $\varphi$ agree on an open set, $\psi=\varphi$ on all of $S$ by the maximum principle.

Reference to Figure~7 shows that $\varphi $ satisfies Neumann conditions on edge $6U$ of $S$, while $\psi $ satisfies Dirichlet conditions on that edge. Now consider a small disk of radius $\varepsilon $ centered at a point in the interior of the boundary edge $6U$. Since $S$ is flat, this disk is isometric to the open half-disk $H = \{(x,y) \in \R^2 : x^2+y^2 < \varepsilon, y \ge 0\}$, so we can view $\varphi$ as an eigenfunction on $H$. Since $\varphi$ satisfies Neumann boundary conditions on the $x$-axis, we can extend $\varphi$ by reflection to an eigenfunction $\varphi_1$ on the open disk $B= \{(x,y) \in \R^2:x^2+y^2 < \varepsilon\}$ that is invariant under reflection in the $x$-axis.  But $\varphi $ also satisfies Dirichlet conditions on the $x$-axis, so we can also continue $\varphi $ to an eigenfunction $\varphi_2$ on $B$ that is anti-invariant under reflection in the $x$-axis.  But $\varphi_1$ and $\varphi_2$ agree on an open set, so $\varphi_1 = \varphi_2$, a contradiction. Thus $M_1$ and $M_2$ are not Dirichlet isospectral, and Theorem~\ref{t2} is proved.
\end{proof}

\section{Inaudible singularities; concluding remarks}\label{InaudSing}

We conclude with a few remarks, references to related results, and open questions.

\begin{remarks} The results and methods above lead to some other interesting phenomena, and highlight several natural questions.
\begin{enumerate}
\item Some of the examples of isospectral surfaces in \cite{BCDS} show that one cannot hear whether there is a singularity in the interior: some of the isospectral pairs have the property that $M_1$ has an interior cone singularity, while $M_2$ does not.

If one constructs isospectral bordered surfaces $M_1$ and $M_2$ using the Gerst Gassmann-Sunada triple $(G,\Gamma _1,\Gamma _2)$ of section~\ref{Construction} with generators $st$, $t$, and $tu$ but using the triangular fundamental tile $T$ in Figure~9, then $M_1$ and $M_2$ exhibit this same phenomenon. The surfaces thus constructed are shown in Figure~9. Both surfaces are flat annuli with polygonal boundary, but $M_2$ has a cone singularity in the interior while $M_1$ does not ($M_1$ is a manifold with corners, while $M_2$ is more singular).  Similarly, if one constructs bordered surfaces $M_1$ and $M_2$ using the triangular tile $T$ but the generators $\sigma $, $t$ and $u$ as in section~\ref{Construction}, then $M_1$ has a single interior cone singularity, while $M_2$ has two cone singularities. Thus one cannot hear the nature of singularities.

\item We note that the content of Theorem~\ref{t1} can be expressed as the assertion that one cannot hear the vanishing of the first Stiefel-Whitney class $w_1$ of a Riemannian manifold with boundary.  It is then natural to ask: for an orientable Riemannian manifold, can one hear the vanishing of the second Stiefel-Whitney class $w_2$?  That is, can one infer from the Laplace spectrum whether or not the manifold admits a Spin structure?  This question was answered negatively by Roberto Miatello and Ricardo Podest\'a in \cite{MiatelloPodesta}.

\item Recall that an \emph{orbifold chart} on a space $X$ (see \cite{ALR}) is given by a connected open subset $\wt{U}$ of $\R^n$, a finite group $G$ of diffeomorphisms of $U$, and a $G$-invariant map $\varphi:\wt{U}\to X$ that induces a homeomorphism of the orbit space $G\backslash\wt{U}$ with an open subset of $X$; an orbifold is then a space equipped with a cover of orbifold charts that satisfy a suitable compatibility condition.  An orbifold is \emph{locally orientable} if in each such orbifold chart, the action of the group $G$ on $\wt{U}$ is by orientation-preserving diffeomorphisms of $\wt{U}$.  In contrast with our main result, it was recently shown by  Sean Richardson and Elizabeth Stanhope \cite{RichardsonStanhope} that one \emph{can} hear local orientability of a Riemannian orbifold.
\end{enumerate}
\end{remarks}

Finally, we note that our main result and the example depicted in Figure~9 call attention to two open problems:
\begin{itemize}
\item Can a \emph{closed} orientable Riemannian manifold be isospectral to a \emph{closed} nonorientable Riemannian manifold?

\item Can a Riemannian orbifold with nonempty singular set be isospectral to a Riemannian manifold?
\end{itemize}


\begin{center}
\includegraphics[height=2.25in]{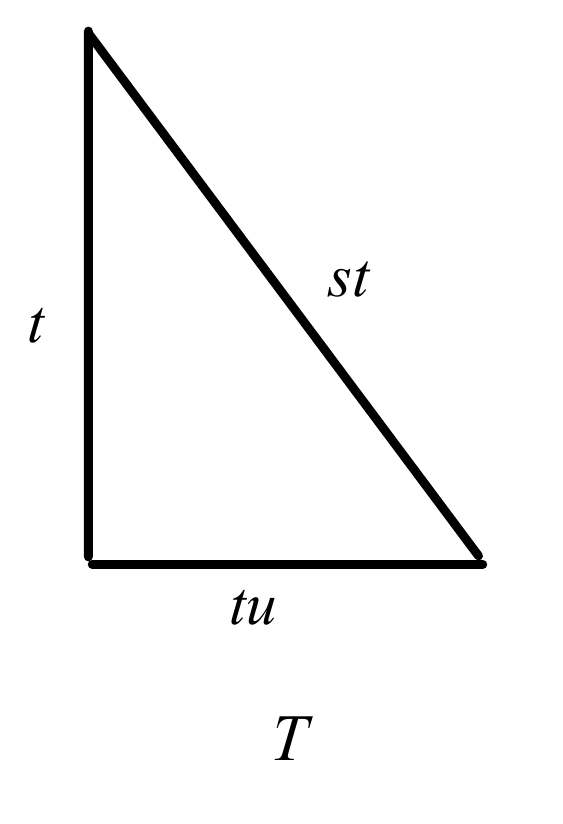}
\includegraphics[width=4.5in]{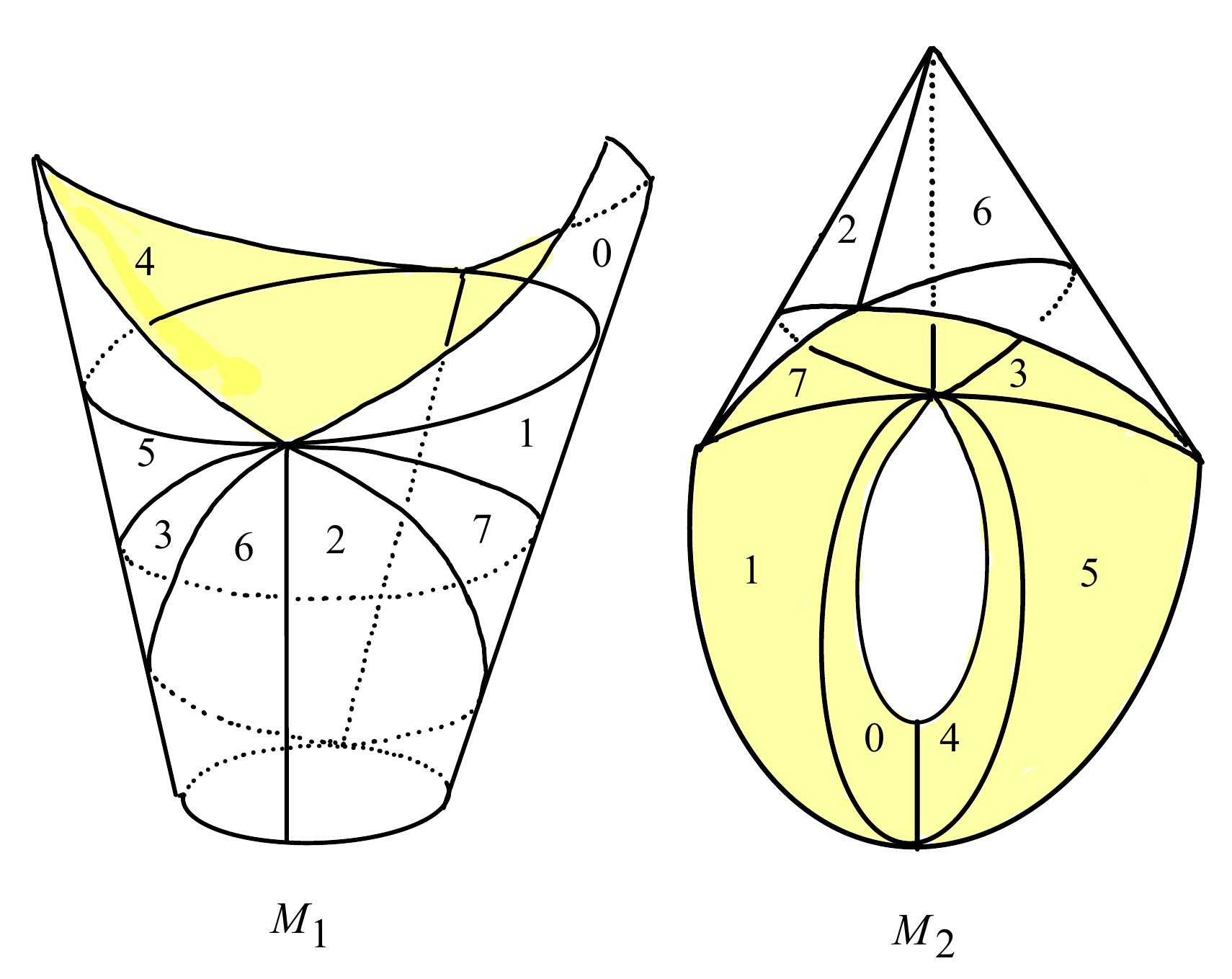}
{\it Figure 9}
\end{center}

\vfill
\begin{flushleft}
\footnotesize
Pierre B\'erard\\
Universit\'{e} Grenoble Alpes and CNRS\\
Institut Fourier, CS 40700\\ 38058 Grenoble cedex 9, France.\\
\smallskip
email:pierrehberard@gmail.com\\
\medskip
David L. Webb\\
Darmouth College\\
Hanover, NH 03755 (USA)\\
\smallskip
email:David.L.Webb@Dartmouth.edu
\end{flushleft}


\newpage

\begin{thebibliography}{10}

\bibitem{ALR} Adem, A., Leida, J., and Ruan, Y.
\newblock {\it Orbifolds and stringy topology.}
\newblock Cambridge Tracts in Mathematics \textbf{171}.
\newblock Cambridge University Press, Cambridge, 2007.

\bibitem{Be1} B\'{e}rard, P.
\newblock {\it Spectral geometry: direct and inverse problems.}
\newblock Lecture Notes in Mathematics, vol. \textbf{1207}.
\newblock Springer, Berlin Heidelberg New York, 1986.

\bibitem{Be2} B\'{e}rard, P.
\newblock Transplantation et isospectralit\'{e} I.
\newblock {\it Math. Ann.} \textbf{292} (1992), 547--559.

\bibitem{Be3} B\'{e}rard, P.
\newblock  Domaines plans isospectraux \`{a} la Gordon-Webb-Wolpert: une
    preuve \'{e}l\'{e}mentaire.
\newblock {\it Afrika Matematika} \textbf{3} (1) (1992), 135--146.  Also {\it S\'eminaire de Th\'eorie Spectrale et G\'eom\'etrie}, Univ. Grenoble I, Saint-Martin-d'H\`eres \textbf{10} (1991–1992), 131–-142.

\bibitem{BerardPesce} B\'erard, P., and Pesce, H.
\newblock Construction de vari\'et\'es isospectrales autour du th\'eor\`eme de T.~Sunada. (French) [Construction of isospectral manifolds using the theorem of T. Sunada]
\newblock {\it Progress in inverse spectral geometry}, 63--83, Trends Math., Birkh\"auser, Basel, 1997.

\bibitem{BW} B\'{e}rard, P. and Webb, D.
\newblock On ne peut pas entendre l'orientabilit\'e d'une surface.
\newblock  {\it C. R. Acad. Sci. Paris S\'er. I Math.} \textbf{320} (1995), no. 5, 533--536.

\bibitem{BW1} B\'{e}rard, P. and Webb, D.
\newblock One can't hear orientability of surfaces.
\newblock  {\it MSRI Preprint} No.~\textbf{005-95} (1995), 24 p.

\bibitem{Brooks} Brooks, R.
\newblock The Sunada Method.
\newblock {\it Tel Aviv Topology Conference: Rothenberg Festschrift} (1998), 25--35,
{\it Contemp. Math.} \textbf{231}, Amer. Math. Soc., Providence, RI, 1999.

\bibitem{BrooksMonthly} Brooks, R.
\newblock Constructing isospectral manifolds.
\newblock {\it Amer. Math. Monthly} \textbf{95} (1988), no. 9, 823--839.

\bibitem{Bu1} Buser, P.
\newblock Cayley graphs and isospectral domains.
\newblock In {\it Proc. Taniguchi Symp. Geometry and Analysis on Manifolds}, Sunada, T. (ed.) 1987.
\newblock Lecture Notes in Mathematics, vol. \textbf{1339}, 64--77.
\newblock Springer, Berlin Heidelberg New York, 1988.

\bibitem{Bu2} Buser, P.
\newblock Isospectral Riemann surfaces.
\newblock  {\it Ann. Inst. Fourier} \textbf{36} (2) (1986), 167--192.

\bibitem{BCDS} Buser, P., Conway, J., Doyle, P., and Semmler, K.-D.
\newblock  Some planar isospectral domains.
\newblock   {\it Internat. Math. Res. Notices} 1994, no. 9, 391ff., approx. 9 pp.

\bibitem{BuserBook} Buser, P.
\newblock Geometry and spectra of Riemann surfaces.
\newblock Birkha\"user, Boston, 1992.

\bibitem{C} Chavel, I.
\newblock {\it Eigenvalues in Riemannian geometry.}
\newblock  Academic Press, London, 1984.

\bibitem{CR} Curtis, C.  and Reiner, I.
\newblock {\it Methods of representation theory with applications to finite groups and orders, vol. I.}
\newblock Wiley, New York, 1981.

\bibitem{DI}
\newblock DeMeyer, F., and Ingraham, E.
\newblock {\it Separable algebras over commutative rings.}
\newblock Lecture Notes in Mathematics, vol. \textbf{181}.
\newblock Springer, Berlin Heidelberg New York, 1971.

\bibitem{DoyleRossetti}
\newblock Doyle, P., and Rossetti, J. P.
\newblock Isospectral hyperbolic surfaces have matching geodesics.
\newblock {\it New York J. Math.} \textbf{14} (2008), 193--204.

\bibitem{FaMa} Farb, B. and Margalit, D.
\newblock {\it A primer on mapping class groups}.
\newblock Princeton University Press, Princeton and Oxford, 2011.

\bibitem{Ga} Gassmann, F.
\newblock Bemerkung zu der vorstehenden Arbeit von Hurwitz.
\newblock {\it Math. Z.} \textbf{25} (1926), 124--143.

\bibitem{Ge} Gerst, I.
\newblock On the theory on $n$-th power residues and a conjecture of Kronecker.
\newblock {\it Acta Arithmetica} \textbf{17} (1970), 121--139.

\bibitem{Sunada20YrsLater} Gordon, C.
\newblock Sunada's isospectrality technique: two decades later.
\newblock {\it Spectral analysis in geometry and number theory}, 45--58, {\it Contemp. Math.}, \textbf{484}, Amer. Math. Soc., Providence, RI, 2009.

\bibitem{Gu} Guralnick, R.
\newblock Subgroups inducing the same permutation representation.
\newblock {\it J. Algebra} \textbf{81} (1983), 312--319.

\bibitem{GWW} Gordon, C., Webb, D. and Wolpert, S.
\newblock Isospectral plane domains and surfaces via Riemannian orbifolds.
\newblock {\it Invent. Math.} \textbf{110} (1992), 1--22.

\bibitem{Herbrich} Herbrich, P.
\newblock On inaudible properties of broken drums --- Isospectrality with mixed Dirichlet-Neumann boundary conditions.
\newblock  arXiv:1111.6789v3.

\bibitem{Hubbard} Hubbard, J.
\newblock {\it Teichm\"uller theory and applications to geometry, topology, and dynamics,} vol. 1.
\newblock Matrix Editions, Ithaca, New York, 2006.

\bibitem{JamesLiebeck} James, G. and Liebeck, M.
\newblock {\it Representations and characters of groups.}
\newblock Cambridge University Press, Cambridge, 1993.

\bibitem{K} Kac, M.
\newblock Can one hear the shape of a drum?.
\newblock {\it Amer. Math. Monthly} \textbf{73} (1966), 1--23.

\bibitem{L} Lenstra, H.
\newblock Grothendieck groups of abelian group rings.
\newblock {\it J. Pure Appl. Algebra} \textbf{20} (1981), 173--193.

\bibitem{MiatelloPodesta} Miatello, R. and Podest\'a, R.
\newblock Spin structures and spectra of $\Z_2^{\;k}$-manifolds.
\newblock {\it  Math. Z.} \textbf{247} (2004), no. 2, 319--335.

\bibitem{PP} Papadopoulos, A. and Penner, R.
\newblock Hyperbolic metrics, measured foliations and pants de\-com\-positions for non-orientable surfaces.
\newblock {\it Asian J. Math.} \textbf{20} (2016), no. 1, 157--182.

\bibitem{P} Perlis, R.
\newblock On the equation $\zeta _K(s) = \zeta _{K'}(s)$.
\newblock  {\it J. Number Theory} \textbf{9} (1977), 342--360.

\bibitem{RichardsonStanhope} Richardson, S. and Stanhope, E.
\newblock You can hear the local orientability of an orbifold.
\newblock {\it Differential Geom. Appl.} \textbf{68} (2020), 101577, 7 pp.

\bibitem{Sc} Scott, P.
\newblock The geometries of 3-manifolds.
\newblock  {\it Bull. London Math. Soc.} \textbf{15} (1983), 401--487.

\bibitem{Se} Serre, J.-P.
\newblock {\it Linear representations of finite groups.}
\newblock  Springer, Berlin Heidelberg New York, 1977.

\bibitem{Su} Sunada, T.
\newblock Riemannian coverings and isospectral manifolds.
\newblock  {\it Ann. Math.} \textbf{121} (1985), 248--277.

\bibitem{T} Thurston, W.
\newblock {\it The geometry and topology of 3-manifolds.}
\newblock Mimeographed lecture notes.
\newblock  Princeton University, 1976--79.

\end{thebibliography}
\end{document}